\documentclass[11pt]{article}
\usepackage[english,activeacute]{babel}
  \usepackage[T1]{fontenc}
\usepackage{subfig,psfrag}
\usepackage[percent]{overpic}
\usepackage{graphicx,amsmath,amssymb,amsthm,xcolor,enumitem}
\usepackage{upgreek}
\usepackage[active]{srcltx}
\usepackage{bbm}
\usepackage{esint}

 \newcommand{\m}{\boldsymbol{m}}
 \newcommand{\QQ}{\boldsymbol{Q}}
 \newcommand{\n}{\boldsymbol{n}}
  \newcommand{\f}{\boldsymbol{f}}

\newcommand{\cc}{\mathrm{c}}
\newcommand{\XX}{\boldsymbol{X}}

\newcommand{\Z}{\mathbb{Z}}

\renewcommand{\c}{\mathfrak{c}}
\renewcommand{\S}{\mathbb{S}}

\renewcommand{\P}{\mathcal{P}}

\newcommand{\R}{\mathbb{R}}
\newcommand{\C}{\mathbb{C}}
\newcommand{\N}{\mathbb{N}}

\newcommand{\boR}{\mathcal{R}}
\newcommand{\boS}{\mathcal{S}}

\newcommand{\boT}{\mathcal{T}}
\newcommand{\boB}{\mathcal{B}}
\newcommand{\boD}{\mathcal{D}}

\newcommand{\ptl}{{\partial}}

 \providecommand{\abs}[1]{|#1 |}
\newcommand{\wh}{\widehat }

\newcommand{\vp}{\varphi }

 \newcommand{\BMO}{\textup{BMO}}
 \newcommand{\loc}{\mathrm{loc}}
 \newcommand{\uloc}{\mathrm{uloc}}
 \providecommand{\norm}[1]{\|#1 \|}

\renewcommand{\div}{\operatorname{div}}
\renewcommand{\Re}{\operatorname{Re}}
\renewcommand{\Im}{\operatorname{Im}}

\newcommand{\Erf}{\operatorname{Erf}}

\newcommand{\bqq}{\begin{equation*}}
\newcommand{\eqq}{\end{equation*}}
\newcommand{\bq}{\begin{equation}}
\newcommand{\eq}{\end{equation}}
\newcommand{\A}{\boldsymbol{A}}

\renewcommand{\H}{\boldsymbol{H}}

\newcommand{\ve}{\varepsilon }

\newtheorem{thm}{Theorem}[section]
\newtheorem{definition}[thm]{Definition}
\newtheorem{prop}[thm]{Proposition}

\newtheorem{lemma}[thm]{Lemma}
\newtheorem{corollary}[thm]{Corollary}
\newtheorem{cor}[thm]{Corollary}

  \renewcommand{\S}{\mathbb{S}}

\theoremstyle{definition}
\newtheorem{remark}[thm]{Remark}
\newtheorem*{thank}{Acknowledgments}
\newcommand{\grad}{\nabla}

\begin{document}

\title{The Cauchy problem for the Landau--Lifshitz--Gilbert equation in
BMO and self-similar solutions}

\author{
\renewcommand{\thefootnote}{\arabic{footnote}}
Susana Guti\'errez\footnotemark[1] ~and Andr\'e de Laire\footnotemark[2]}
\footnotetext[1]{School of Mathematics,
University of Birmingham, Edgbaston, Birmingham, B15 2TT, United
Kingdom. E-mail: {\tt s.gutierrez@bham.ac.uk}}
\footnotetext[2]{
	Univ.\ Lille, CNRS, Inria, UMR 8524, Laboratoire Paul Painlev\'e, F-59000 Lille, France.\\
	E-mail: {\tt andre.de-laire@univ-lille.fr}}

\date{}
\maketitle

\begin{abstract}
We prove a global well-posedness result for the Landau--Lifshitz equation with Gilbert damping provided that
the BMO semi-norm of the initial data is small. As a consequence, we deduce the existence of self-similar solutions in any dimension.
In the one-dimensional case, we characterize the self-similar solutions associated with an initial data given by some
($\S^2$-valued) step function and establish  their stability.
We also show the existence of multiple solutions if the damping is strong enough.

Our arguments rely on the study of a dissipative quasilinear Schr\"odinger equation obtained via the stereographic projection
and techniques introduced by Koch and Tataru.

\medskip
\noindent{{\em Keywords and phrases:} Landau--Lifshitz--Gilbert
equation, global well-posedness, discontinuous initial data,  stability, self-similar
solutions, dissipative Schr\"odinger equation, complex Ginzburg--Landau equation,
ferromagnetic spin chain, heat-flow for harmonic maps.
}

\medskip
 \noindent{2010 \em{Mathematics Subject Classification}:} 
35R05, 
35Q60, 
35A01, 
35C06, 
35B35, 
35Q55, 
35Q56, 
35A02, 
53C44. 
%

\end{abstract}
\tableofcontents

\section{Introduction and main results}

\setcounter{equation}{0}
\numberwithin{equation}{section}

\setcounter{footnote}{2}
\renewcommand{\thefootnote}{\arabic{footnote}}

We consider the
Landau--Lifshitz--Gilbert (LLG) equation
\begin{equation}\label{LLG}
\ptl_t \m=\beta \m\times \Delta \m -\alpha \m \times (\m\times \Delta \m),
\quad\text{ on }\R^N\times\R^+,
 \tag{LLG$_\alpha$}
\end{equation}
where  $\m=(m_1, m_2, m_3):\R^N \times \R^+\longrightarrow \S^2$
is the spin vector, $\beta\geq 0$, $\alpha\geq 0$$, \times$ denotes
the usual cross-product in $\R^3$, and $\mathbb{S}^2$ is the unit
sphere in $\mathbb{R}^3$. This model introduced by Landau and Lifshitz describes the dynamics for the spin in ferromagnetic materials
\cite{landaulifshitz,gilbert}
and constitutes a fundamental equation in the magnetic recording industry \cite{wei2012}. The parameters $\beta\geq 0$ and $\alpha\geq 0$ are respectively the so-called exchange constant and Gilbert damping, and take into account the exchange of energy in the system and the effect of damping on the spin chain.
Note that, by performing a time-scaling, we assume w.l.o.g.\ that
$$
 \alpha\in[0,1]\quad \text{and}\quad  \beta=\sqrt{1-\alpha^2}.
$$ 
The Landau--Lifshitz family of equations includes as special cases the well-known heat-flow
for harmonic maps and the Schr\"odinger map equation onto the 2-sphere. 
In the limit case $\beta = 0$ (and so $\alpha=1$) the LLG equation reduces to the heat-flow equation for
harmonic maps
\bq\label{HFHM}\tag{HFHM}
\ptl_t \m-\Delta\m=\abs{\grad{\m}}^2\m ,\quad\text{ on }\R^N\times\R^+.
\eq
The case when $\alpha=0$ (i.e.\ no dissipation/damping) corresponds to the Schr\"odinger map equation
\bq\label{SM}\tag{SM}
\ptl_t \m= \m\times \Delta \m,\quad\text{ on }\R^N\times\R^+.
\eq
In the one-dimensional case $N=1$, we established in \cite{gutierrez-delaire} the existence and
asymptotics of the family $\{\m_{c,\alpha}\}_{c>0}$ of self-similar solutions of \eqref{LLG} for any fixed $\alpha\in[0,1]$,
extending the results in Guti\'errez, Rivas and Vega \cite{vega-gutierrez} in the setting of the Schr\"odinger map equation and related
binormal flow equation.
The  motivation for the results presented in this paper first originated from the desire to study further properties of the
self-similar solutions found in \cite{gutierrez-delaire}, and in particular their stability.
In the case $\alpha=0$, the stability of the self-similar solutions of the Schr\"odinger map has been considered in the series of papers by Banica and Vega \cite{banica-vega,Banica-Vega-2,banica-vega-Arch},
but no stability result is known  for these solutions in the presence of damping, i.e. $\alpha>0$. One of the key ingredients in the analysis given by Banica and Vega is the reversibility in time of the equation in the absence of damping. However, since \eqref{LLG} is a dissipative equation for $\alpha>0$, this property is no longer available and a new approach is needed.

In the one-dimensional case and for fixed $\alpha\in [0,1]$, the self-similar solutions of \eqref{LLG} constitute a uniparametric family $\{ \m_{c,\alpha} \}_{c>0}$ where $\m_{c,\alpha} $ is defined by
$$\m_{c,\alpha}(x,t)=\f\left( \frac{x}{\sqrt{t}}\right),$$
for some profile $\f:\mathbb{R}\longrightarrow  \S^2$, and is associated with an initial condition given by a step function (at least when $c$ is small) of the form
 \bq\label{intro-def-m0-c}
 \m^0_{c,\alpha}:=\A^+_{c,\alpha}\chi_{\R^+}+\A^-_{c,\alpha}\chi_{\R^-},
\eq
where $\A^{\pm}_{c,\alpha}$ are certain unitary vectors
and $\chi_E$  denotes the characteristic function of a set $E$. In particular, when $\alpha>0$, the Dirichlet energy associated with the solutions $\m_{c,\alpha}$ given by
\bq\label{diverge}
\norm{\grad\m_{c,\alpha}(\cdot,t)}^2_{L^2}=c^2\Big( \frac{2\pi}{\alpha t} \Big)^{1/2},\qquad  t>0,
\eq
diverges as $t\to 0^+$\footnote{We refer the reader to Theorem~\ref{thm-self} in the Appendix and to
\cite{gutierrez-delaire} for precise statements of these results.}. 

A first natural question in the study of the stability properties of the family of solutions $\{ \m_{c,\alpha}\}_{c>0}$ is whether or not it is possible to develop a well-posedness theory for the Cauchy problem for \eqref{LLG} in a functional framework that allows us to handle initial conditions of the type \eqref{intro-def-m0-c}. In view of \eqref{intro-def-m0-c} and \eqref{diverge}, such a framework should allow some ``rough'' functions 
(i.e.\ function spaces beyond the ``classical'' energy ones) and step functions.

A few remarks about previously known results in this setting are in order. 
In the case $\alpha>0$,  global well-posedness results for \eqref{LLG} have been established in $N\geq 2$ by
Melcher~\cite{melcher} and by Lin, Lai and Wang~\cite{lin-lai-wang}
for initial conditions with a smallness condition on the gradient in the $L^N(\R^N)$ and the Morrey $M^{2,2}(\R^N)$ norm\footnote{See footnote in Section~\ref{sec-singular} for the definition of the Morrey space $M^{2,2}(\R^N)$.}, respectively.
Therefore these results do not apply to the initial condition $\m^0_{c,\alpha}$.
 When $\alpha=1$, global well-posedness results for the heat flow for harmonic maps \eqref{HFHM} have been obtained by Koch and Lamm~\cite{koch-lamm}  for an initial condition $L^\infty$-close to a point and improved to an initial data with small
 BMO semi-norm  by Wang~\cite{wang}. The ideas used in \cite{koch-lamm} and \cite{wang} rely on techniques introduced by
 Koch and Tataru~\cite{koch-tataru} for the Navier--Stokes equation. Since  $\m^0_{c,\alpha}$
  has a small BMO semi-norm if $c$ is small,  the results in \cite{wang} apply to the case $\alpha=1$.

There are two main purposes in this paper. The first one is to adapt and extend the techniques developed
in \cite{koch-lamm,koch-tataru,wang} to prove a global well-posedness result for \eqref{LLG}
with $\alpha\in(0,1]$  for data $\m^0$ in $L^{\infty}(\R^N; \S^2)$ with small BMO semi-norm.
The second one is to apply this result to establish the stability  of the family of self-similar solutions $\{\m_{c,\alpha}\}_{c>0}$ found in \cite{gutierrez-delaire}
and derive further properties for these solutions. In particular, a further understanding of the properties of the  functions $\m_{c,\alpha}$
will allow us to prove the existence of multiple smooth solutions of \eqref{LLG} associated with the same initial
condition, provided that $\alpha$ is close to one.

In order to state the first of our results, we introduce the function space $X$ as follows:
\begin{equation*}
 X=\{\, v:\mathbb{R}^n\times \mathbb{R}^{+}\to \mathbb{R}^3  :
 v, \nabla v\in L^{1}_{\loc} (\mathbb{R}^N\times \mathbb{R}^{+}) \, \text{ and }\,
                   {\| v\|}_{X}:=\text{sup}_{t>0}{\| v(t)\|}_{L^{\infty}} +[v]_{X}< \infty\}
\end{equation*}
where
$$
 [v]_{X}:=\text{sup}_{t>0} \sqrt{t} {\|\nabla v\|}_{L^\infty} +
  \sup_{\substack{ x\in \R^N\\ r>0}} \left( \frac{1}{r^N} \int_{B_r(x)\times [0, r^2]} \abs{\nabla v(y,t)}^2\,  dt \,dy \right)^{\frac{1}{2}},
$$
and $B_r(x)$ denotes the ball with center $x$ and radius $r>0$ in $\mathbb{R}^{N}$. 
Let us remark that the first term in the definition of $[v]_{X}$ allows to capture a blow-up rate of $1/\sqrt{t}$ for  $ {\| \nabla v(t)\|}_{L^{\infty}}$, as $t\to 0^+$.  This is exactly the blow-up rate for the self-similar solutions (see~\eqref{def-self-sim} and \eqref{cond-2}). The integral term in $[v]_X$ is associated with the space BMO as explained in Subsection~\ref{subsec-global}, and it is also well adapted to the self-similar solutions (see Proposition~\ref{cotas-self} and its proof).

We can now state the following (global) well-posedness result for the Cauchy problem for the LLG equation:

\begin{thm}\label{thm-cauchy-intro}
Let  $\alpha\in(0,1]$. There exist  constants $M_1,M_2,M_3>0$,
depending only on $\alpha$ and $N$ such that the following holds.
For any
$\m^0\in L^\infty(\R^N;\S^2)$,
$\QQ\in \S^2$, $\delta\in (0,2]$
and $\ve_0>0$ such that $\ve_0\leq M_1\delta^6$,
\bq\label{CI-intro}
\inf_{\R^N}\abs{\m^0-\QQ}^2\geq 2\delta \quad\textup{ and }\quad[\m^0]_{BMO}\leq \ve_0,
\eq
there exists a unique solution $\m\in X(\R^N\times \R^+;\S^2)$ of \eqref{LLG}
with initial condition $\m^0$ such that
\bq
 \label{cond-thm-intro}
 \inf_{\substack{x\in \R^N\\ t>0}} \abs{\m(x,t)-\QQ}^2\geq \frac{4}{1+M_2^2(M_3\delta^4+\delta^{-1})^2}
 \quad \textup{ and }\quad [\m]_{X}\leq 4M_2(M_3\delta^4+8\delta^{-2}\ve_0).
 \eq
In addition, $\m $ is a smooth function belonging to  $C^{\infty}(\R^N\times \R^{+}; \mathbb{S}^2)$.
Furthermore, assume that $\n$ is a solution to \eqref{LLG}
fulfilling \eqref{cond-thm-intro},  with initial condition $\n^0$
 satisfying \eqref{CI-intro}. Then
\bq\label{cont-intro}
 \norm{\m-\n}_{X}\leq \frac{120M_2}{\delta^{2}}\norm{\m^0-\n^0}_{L^\infty}.
\eq
\end{thm}

As we will see in Section~\ref{sec-cauchy}, the proof of Theorem~\ref{thm-cauchy-intro} relies on the use of the stereographic projection
to reduce Theorem~\ref{thm-cauchy-intro} to establish a well-posedness result for the associated dissipative (quasilinear) Schr\"odinger equation (see Theorem~\ref{thm:cauchy}).
In order to be able to apply Theorem~\ref{thm-cauchy-intro} to the study of both the initial value problem related to the LLG equation with a jump initial
condition, and the stability of the self-similar solutions found in \cite{gutierrez-delaire}, we will need a more quantitative version of this result.
A more refined version of Theorem~\ref{thm-cauchy-intro} will be stated in Theorem~\ref{thm:cauchy-LLG} in Subsection~\ref{sec-LLG}.

Theorem~\ref{thm-cauchy-intro} (or more precisely Theorem~\ref{thm:cauchy-LLG}) has two important consequences for the Cauchy problem related to \eqref{LLG} in one dimension:
\begin{eqnarray}
\label{ivp}
\left\{
 \begin{array}{ll}
  \ptl_t \m =\beta \m\times \partial_{xx}\m-\alpha \m \times (\m\times \partial_{xx}\m),\quad & \text{on}\quad \mathbb{R}\times \mathbb{R}^{+},
     \\[2ex]
  \m^0_{\A^{\pm}}:=\A^+\chi_{\R^+}+\A^-\chi_{\R^-},
 \end{array}
\right.
\end{eqnarray}
where $\A^{\pm}$ are two given unitary vectors such that the angle between $\A^{+}$ and $\A^{-}$ is sufficiently small:
\begin{itemize}
\item[\it{(a)}] From the uniqueness statement in Theorem~\ref{thm-cauchy-intro}, we can deduce that the solution to \eqref{ivp} provided by Theorem~\ref{thm-cauchy-intro} is a rotation of a self-similar solution $\m_{c,\alpha}$ for an appropriate value of $c$ (see Theorem~\ref{thm-small-angle} for a precise statement).

\item[\it{(b)}] (\emph{Stability}) From the dependence of the solution with respect to the initial data established in \eqref{cont-intro} and the analysis of the $1d$-self-similar solutions $\m_{c,\alpha}$ carried out in \cite{gutierrez-delaire}, we obtain the following stability result: For any given $\m^0\in \S^2$ satisfying \eqref{CI-intro} and close enough to  $\m^0_{\A^{\pm}}$, 
the solution $\m$ of \eqref{LLG} associated with  $\m^0$ given by Theorem~\ref{thm-cauchy-intro} must remain close to a rotation of a self-similar solution $\m_{c,\alpha}$, for some $c>0$. In particular, $\m$ remains close to a self-similar solution.
\end{itemize}
The precise statement is provided in the following theorem.

\begin{thm}\label{thm:stability}
Let $\alpha\in (0,1]$. There exist constants $L_1,L_2,L_3>0$, $\delta^*\in (-1,0)$,
$\vartheta^*>0$ such that the following holds.
Let $\A^+$,  $\A^- \in \S^2$ with angle $\vartheta$ between them. If
$$0<\vartheta\leq \vartheta^*,$$
then there is $c>0$ such that for every
$\m^0$ satisfying
\bqq
\norm{\m^0-\m^0_{\A^\pm} }_{L^\infty}\leq \frac{c\sqrt{\pi}}{2\sqrt{\alpha}},
\eqq
there exists $\boR\in SO(3)$, depending only on $\A^+$, $\A^-$, $\alpha$ and $c$, such that
there is a unique global smooth solution $\m$ of \eqref{LLG} with initial condition $\m^0$
that satisfies
\bq\label{condition}
  \inf_{\substack{x\in \R\\ t>0}} (\boR \m)_3(x,t)\geq \delta^*
 \quad \textup{ and }\quad [\m]_{X}\leq L_1+L_2 c.
 \eq
Moreover,
 \bqq
\norm{ \m-\boR\m_{c,\alpha}}_{X}\leq  L_3\norm{\m^0-\m^0_{\A^\pm} }_{L^\infty}.
\eqq
In particular,
\bqq
\norm{\partial_x \m-\partial_x \boR\m_{c,\alpha}}_{L^\infty}\leq  \frac{L_3}{\sqrt{ t}}\norm{\m^0-\m^0_{\A^\pm} }_{L^\infty},
\eqq
for all $t>0$.
\end{thm}
Notice that Theorem \ref{thm:stability} provides the existence of a unique solution in the set defined by the conditions \eqref{condition}, and hence it does not exclude the possibility of the existence of other solutions not satisfying these conditions. In fact, as we will see in Theorem~\ref{thm-non-uniq} below, one can prove the existence of multiple solutions of the initial value problem \eqref{ivp}, at least in the case when $\alpha$ is close to $1$.

We point out that our results are valid only for $\alpha>0$. If we let $\alpha\to0$, then
the constants $M_1$ and $M_3$ in Theorem~\ref{thm-cauchy-intro} go to 0 and $M_2$ blows up.
Indeed, we use that the kernel associated with the Ginzburg--Landau semigroup $e^{(\alpha+i\beta)t\Delta}$ belongs
to $L^1$ and its exponential decay. Therefore our techniques cannot be generalized (in a simple way) to cover the critical
case $\alpha=0$. In particular, we cannot recover the stability results for the
self-similar solutions in the case of Schr\"odinger maps proved by Banica and Vega in
\cite{banica-vega,Banica-Vega-2,banica-vega-Arch}.

As mentioned before, in \cite{lin-lai-wang} and \cite{melcher} some global well-posedness results for \eqref{LLG} with $\alpha\in (0,1]$ were proved for initial conditions with small gradient in $L^N(\R^N)$ and $M^{2,2}(\R^N)$, respectively (see footnote in Subsection~\ref{sec-singular} for the definition of the space $M^{2,2}(\R^N)$). 
In view of the embeddings 
$$L^N(\R^N)\subset M^{2,2}(\R^N)\subset BMO^{-1}(\R^N),$$
for $N\geq2$, Theorem~\ref{thm-cauchy-intro} can be seen as generalization of these results
since it covers the case of less regular initial conditions.
The arguments in \cite{lin-lai-wang,melcher} are based on the method of moving frames that produces a covariant
complex Ginzburg--Landau equation. In Subsection~\ref{sec-singular} we give more details and discuss the corresponding
equation in the one-dimensional case and provide some properties related to the self-similar solutions.

Our existence and uniqueness result given by Theorem~\ref{thm-cauchy-intro} requires the initial condition to be small in the 
BMO semi-norm. Without this condition, the solution 
could develop a singularity in finite time. In fact, in dimensions $N=3,4$, Ding and Wang \cite{ding-wang} have proved that
for some smooth  initial conditions with small (Dirichlet) energy, the associated solutions of \eqref{LLG}  blow up in 
finite time.

In the context of the initial value problem \eqref{ivp}, 
the smallness condition in the BMO semi-norm is equivalent to the smallness of the angle between $\A^{+}$ and $\A^{-}$. 
As discussed in \cite{gutierrez-delaire}, in the one dimensional case $N=1$ for fixed $\alpha\in(0,1]$ there is some numerical evidence that
indicates the existence of multiple (self-similar) solutions associated with the same initial condition of the type in \eqref{ivp} (see Figures~2 and 3 in \cite{gutierrez-delaire}). This suggests that the Cauchy problem for \eqref{LLG} with initial condition \eqref{ivp} is ill-posed for general $\A^{+}$ and $\A^{-}$ unitary vectors.

The following result states that in the case when $\alpha$ is close to $1$, one can actually prove
the existence of multiple smooth solutions associated with the same initial condition $\m^{0}_{\A^{\pm}}$. Moreover, given any angle $\vartheta\in(0,\pi)$ between two vectors $\A^{+}$ and $\A^{-}\in\mathbb{S}^2$, one can generate any number of distinct solutions by considering values of $\alpha$ sufficiently close to $1$.
\begin{thm}\label{thm-non-uniq}
Let $k\in \N$, $\A^+$, $\A^-\in \S^2$ and let $\vartheta$ be the angle between $\A^+$ and $\A^-$.
If $\vartheta\in (0,\pi)$, then there exists $\alpha_k \in (0,1)$ such that for every $\alpha\in [\alpha_k,1]$
there are at least $k$ distinct smooth self-similar solutions $\{\m_j\}_{j=1}^k$ in $X(\R\times \R^+;\S^2)$ of \eqref{LLG}
with initial condition $\m^0_{\A^\pm}$. These solutions are characterized by a  strictly increasing sequence
of values $\{c_j\}_{j=1}^k$, with $c_k\to\infty$ as $k\to\infty$, such that
\bq\label{mk-1}
\m_j=\boR_j \m_{c_j,\alpha},
\eq
where $\boR_j \in SO(3)$. In particular
\bq\label{derivada-mk}
\sqrt{t}\norm{\partial_x\m_j(\cdot, t)}_{L^\infty}=c_j, \quad  \text{for all } t>0.
\eq
Furthermore, if $\alpha=1$ and  $\vartheta\in [0,\pi]$, then
there is an infinite number of distinct smooth self-similar solutions  $\{\m_j\}_{j\geq 1}$ in $X(\R\times \R^+;\S^2)$ of \eqref{LLG}
with initial condition $\m^0_{\A^\pm}$. These solutions are also characterized by a sequence  
$\{c_j\}_{j=1}^\infty$  
such that \eqref{mk-1} and \eqref{derivada-mk} are satisfied. This sequence is explicitly  given by
\begin{equation}\label{formula-c-j}
 c_{2\ell+1}= \ell\sqrt{\pi}-\frac{\vartheta}{2\sqrt{\pi} },\quad 
 c_{2\ell}= \ell\sqrt{\pi}+\frac{\vartheta}{2\sqrt{\pi} },  \quad \text{ for }\ell\geq 0.
\end{equation}
\end{thm}

It is important to remark that in particular Theorem~\ref{thm-non-uniq} asserts that when $\alpha=1$,
given $\A^+,\A^-\in\mathbb{S}^2$ such that $\A^+=\A^-$,  there exists an infinite number of distinct solutions $\{\m_j\}_{j\geq 1}$ in $X(\mathbb{R}\times\mathbb{R}^{+}; \mathbb{S}^2)$ of \eqref{LLG} with initial condition $\m^0_{\A^\pm}$ such that $[\m^0_{\A^{\pm}}]_{BMO}=0$. This particular case shows that a condition on the size of $X$-norm of the solution as that given in \eqref{cond-thm-intro} in Theorem~\ref{thm-cauchy-intro} 
is necessary for the uniqueness of solution. 
We recall that for finite energy solutions of \eqref{HFHM} there are several nonuniqueness results  based on Coron's technique \cite{coron}
in dimension $N=3$. Alouges and Soyeur~\cite{alouges-soyeur} successfully adapted this idea to prove the existence of multiple solutions of the (LLG$_\alpha$), with $\alpha>0$,  for maps $ {\boldsymbol{m}}:\Omega\longrightarrow \mathbb{S}^2$, with $\Omega$ a bounded regular domain of $\mathbb{R}^3$. In our case, since  $\{c_j\}_{j=1}^k$ is strictly increasing, we have at least $k$ genuinely different \emph{smooth}
solutions. Notice also that  the identity \eqref{derivada-mk} implies that the $X$-norm of the solution is large as $j\rightarrow\infty$.

\noindent \textbf{Structure of the paper.}
This paper is organized as follows: in Section~\ref{sec-cauchy} we use the ste\-reo\-graphic projection to
reduce matters to the study the initial value problem for the resulting dissipative  Schr\"odinger equation,
prove its global well-posedness in well-adapted  normed spaces,
and use this result to establish Theorem~\ref{thm:cauchy-LLG} (a more quantitative version of Theorem~\ref{thm-cauchy-intro}).
In Section~\ref{sec-self-similar} we focus on  the self-similar solutions and we prove Theorems~\ref{thm:stability}
and \ref{thm-non-uniq}. In Section~\ref{sec-singular} we discuss some implications of the existence of explicit
self-similar solutions for the Schr\"odinger equation obtained by means of the Hasimoto transformation. Finally, and for the convenience of the reader, we have
included some regularity results for the complex Ginzburg--Landau equation and some properties of the self-similar solutions $\m_{c,\alpha}$ in the Appendix.

\noindent{\bf Notations.} We write $\R^+=(0,\infty)$. Throughout this paper we will assume that $\alpha\in (0,1]$ and the constants can depend
on $\alpha$. In the proofs  $A\lesssim B$ stands for $A\leq C \, B$ for some constant $C>0$ depending only on $\alpha$ and $N$.
We denote in bold the vector-valued variables.

Since we are interested in $\S^2$-valued functions, with a slightly abuse of notation,
we denote by
$L^\infty(\R^N;\S^2)$ (resp.  $X(\R^N;\S^2)$)  the space of function in $L^\infty(\R^N;\R^3)$
(resp. $X(\R^N;\R^3)$) such that $\abs{\m}$=1 a.e.\ on $\R^N$.

%
\section{The Cauchy problem}
\label{sec-cauchy}
%

\subsection{The Cauchy problem for a dissipative quasilinear Schr\"odinger equation}
\label{subsec-global}
Our approach to study the Cauchy problem for  \eqref{LLG} consists in analyzing the Cauchy problem
for the associated dissipative quasilinear Schr\"odinger equation through the stereographic projection, and then
``transferring'' the results back to the original equation.
To this end, we introduce the stereographic projection from the South Pole
$
\P:\S^2\setminus\{(0,0,-1)\}\to\C$ defined for  by
$$\P(\m)=\frac{m_1+im_2}{1+m_3}.$$
Let $\m$ be a smooth solution of \eqref{LLG} with \mbox{$m_3>-1$}, then its stereographic projection $u=\P(\m)$
satisfies the quasilinear dissipative Schr\"odinger equation (see e.g.\ \cite{lak-nak} for details)
\bq\label{DNLS}\tag{DNLS}
iu_t +(\beta-i\alpha)\Delta u=2(\beta-i\alpha)\frac{\bar u (\grad u)^2}{1+\abs{u}^2}.
\eq
At least formally, the Duhamel formula gives the integral equation:
\bq\label{duhamel}\tag\tag{IDNLS}
u(x,t)=S_{\alpha}(t) u^0+\int_0^t S_\alpha(t-s)g(u)(s)\,ds,
\eq
where $u^0=u(\cdot,0)$   corresponds to the initial condition,
$$g(u)=-2i(\beta-i\alpha)\frac{\bar u (\grad u)^2}{1+\abs{u}^2}$$
and $S_\alpha(t)$ is the dissipative Schr\"odinger semigroup (also called the complex Ginzburg--Landau semigroup) given by
$S_\alpha(t)\phi=e^{(\alpha+i\beta)t\Delta}\phi$, i.e.
\bq\label{semigrupo}
(S_\alpha(t)\phi)(x)=\int_{\R^N}G_\alpha(x-y,t)\phi(y)\,dy, \quad \text{with } \quad G_\alpha(x,t)=\frac{ e^{-\frac{\abs{x}^2}{4(\alpha+i\beta)t} } }{ (4\pi(\alpha+i\beta) t)^{N/2}}.
\eq

One difficulty in studying  \eqref{duhamel} is to handle the term $g(u)$.
Taking into account that
\bq\label{ineq0}
\abs{\beta-i\alpha}=1 \ \text{ and } \
\frac{a}{1+a^2}\leq \frac12, \quad \text{for all }a\geq 0,
\eq
we see that
 \bq\label{est-g}|g(u)|\leq |\grad u|^2,
\eq
so we need to control $|\grad u|^2$. Koch and Taratu dealt with a similar problem when studying the well-posedness for the Navier--Stokes equation in \cite{koch-tataru}.
Their approach was to introduce some new spaces related to BMO  and $\BMO^{-1}$. Later, Koch and Lamm \cite{koch-lamm}
and  Wang \cite{wang} have adapted these spaces to study some geometric flows. Following these ideas, we define the Banach spaces
\begin{align*}
 X(\R^N\times \R^+;F)&=\{v:\R^N\times \R^+\to F \ : \ v,\grad v\in L^1_{\loc}(\R^N\times \R^+), \  \norm{v}_{X}<\infty\}
 \qquad \text{and} \\
 Y(\R^N\times \R^+;F)&=\{v:\R^N\times \R^+\to F \ : v\in L^1_{\loc}(\R^N\times \R^+), \  \norm{v}_{Y}<\infty\},
\end{align*}
where
\begin{align*}
\norm{v}_{X}:=&\sup_{t>0}\norm{v}_{L^\infty}+[v]_X,\qquad \text{with}\\
 [v]_{X}:=&\sup_{t>0} \sqrt{t}\norm{\grad v}_{L^\infty}+
\sup_{\substack{ x\in \R^N\\ r>0}}\left( \frac1{r^N}\int_{Q_r(x)} \abs{\grad v(y,t) }^2\,dt\,dy \right)^\frac12, \qquad\text{and}\\
\norm{v}_{Y}=&\sup_{t>0} t\norm{v}_{L^\infty}+
\sup_{\substack{ x\in \R^N\\ r>0}}\frac1{r^N}\int_{Q_r(x)} \abs{ v(y,t) }\,dt\,dy.
\end{align*}
Here $Q_r(x)$ denotes the parabolic ball $Q_r(x)=B_r(x)\times [0,r^2]$ and $F$ is either $\C$ or $\R^3$.
The absolute value stands for the complex absolute value if $F=\C$
and for the euclidean norm if $F=\R^3$. We denote with the same symbol the absolute value
in $F$ and $F^3$. Here and in the sequel we will omit the domain in the norms and semi-norms
when they are taken in the whole space, for example  $\norm{\cdot}_{L^p}$  stands for
$\norm{\cdot}_{L^p(\R^N)}$, for $p\in[1,\infty]$.

The spaces $X$ and $Y$ are related to the spaces $\BMO(\R^N)$ and $\BMO^{-1}(\R^N)$ and are well-adapted
to study problems involving the heat semigroup $S_1(t)=e^{t\Delta}$. In order to establish the properties of the semigroup $S_\alpha(t)$
with $\alpha\in (0,1]$, we introduce the spaces
$\BMO_\alpha(\R^N)$ and $\BMO_\alpha^{-1}(\R^N)$ as the space of distributions $f\in S'(\R^N;F)$
such that the semi-norm and norm  given respectively  by
\begin{align*}
[f]_{BMO_\alpha}:=&\sup_{\substack{ x\in \R^N\\ r>0}}\left( \frac1{r^N}\int_{Q_r(x)} \abs{\grad S_\alpha(t)f}^2\,dt\,dy \right)^\frac12,\qquad {\hbox {and}} \\
\norm{f}_{BMO^{-1}_\alpha}:=&\sup_{\substack{ x\in \R^N\\ r>0}}\left( \frac1{r^N}\int_{Q_r(x)}\abs{ S_\alpha(t)f}^2\,dt\,dy \right)^\frac12,
\end{align*}
are finite.

On the one hand, the Carleson measure characterization of BMO functions (see \cite[Chapter~4]{stein} and
\cite[Chapter~10]{lemarie})
yields that for fixed $\alpha\in(0,1]$, ${BMO_\alpha}(\R^N)$ coincides with the classical $\BMO(\R^N)$ space\footnote{
$$ BMO(\R^N)=\{f:\R^N\times [0,\infty)\to F \ : f\in L^1_{\loc}(\R^N),  \ [f]_{BMO}<\infty\},$$
with the semi-norm
\bqq
[f]_{BMO}=\sup_{\substack{x\in\R^N\\ r>0}}  \fint_{B_r(x)}\abs{f(y)-f_{x,r}}\,dy,
\eqq
where $f_{x,r}$ is the average $$f_{x,r}=\fint_{B_r(x)} f(y)\,dy=\frac{1}{\abs{B_r(x)}} \int_{B_r(x)} f(y)\,dy.$$}, that is for all $\alpha\in(0,1]$ there exists a constant  $\Lambda>0$ depending only on $\alpha$ and $N$ such that
\bq\label{equiv-BMO}
\Lambda[f]_{BMO}\leq [f]_{BMO_\alpha}\leq \Lambda^{-1}[f]_{BMO}.
\eq
On the other hand, Koch and Tataru proved in \cite{koch-tataru} that $\BMO^{-1}$ (or equivalently $\BMO^{-1}_{1}$, using our notation) can be characterized as the space of derivatives of functions in BMO. A straightforward  generalization of their argument shows that the same result holds for $\BMO_\alpha^{-1}$ (see Theorem~\ref{tataru}). Hence, using the Carleson measure characterization theorem, we conclude that $\BMO_\alpha^{-1}$
coincides with the space  $\BMO^{-1}$ and that there exists a constant  $\tilde\Lambda>0$, depending only on $\alpha$ and $N$, such that
 \bq\label{equiv-norm}
 \tilde\Lambda\norm{f}_{BMO^{-1}}\leq  \norm{f}_{BMO_\alpha^{-1}} \leq \tilde\Lambda^{-1}\norm{f}_{BMO^{-1}}.
 \eq
The above remarks allows us to use several of the estimates proved in \cite{koch-lamm,koch-tataru,wang} in the case $\alpha=1$
to study the integral equation \eqref{duhamel} by using a fixed-point approach.

Our first result concerns the global well-posedness of the Cauchy problem for \eqref{duhamel} with small
initial data in $\BMO(\R^N)$.
%
\begin{thm}\label{thm:cauchy}
 Let $\alpha\in(0,1]$. There exist constants $C,K\geq 1$ such that for every $L\geq 0$, $\ve>0$, and $\rho>0$ satisfying
\bq\label{hyp1}
8C(\rho+\ve)^2\leq \rho,
\eq
if $u^0\in L^\infty(\R^N;\C)$,  with
\bq\label{cond-CI}
\norm{u^0}_{L^\infty}\leq L \quad\textup{ and }\quad [u^0]_{BMO}\leq \ve,
\eq
then there exists a unique  solution $u\in X(\R^N\times \R^+;\C)$ to \eqref{duhamel} such that
\bq\label{cond-small}
[u]_{X}\leq K(\rho+\ve).
\eq
Moreover,
\begin{enumerate}[label=({\roman*}),ref=({\roman*})]
\item $\sup_{t>0}\norm{u}_{L^\infty}\leq K(\rho+L)$.
 \item $u\in C^\infty(\R^N\times \R^+)$ and \eqref{DNLS} holds pointwise.
  \item\label{converge-CII} $\lim\limits_{t\to0^+} u(\cdot,t)=u^0$ as tempered distributions. Moreover,
  for every $\vp\in S(\R^N)$, we have
  \bq\label{conver-en-0}
\norm{ (u(\cdot,t)-u^0)\vp}_{L^1}\to 0, \quad \textup{ as }t\to 0^+.
\eq
 \item (Dependence on the initial data) Assume that $u$ and $v$ are respectively
 solutions to \eqref{duhamel} fulfilling \eqref{cond-small} with initial conditions $u^0$ and $v^0$
 satisfying \eqref{cond-CI}. Then
\bq\label{cont-CI}
 \norm{u-v}_{X}\leq 6K \norm{u^0-v^0}_{L^\infty}.
\eq
\end{enumerate}
\end{thm}
Although  condition \eqref{hyp1} appears naturally from the fixed-point used in the proof, it may be
no so clear at first glance. To better understand it, let us define for $C>0$
\bq\label{set-S}
\mathcal S(C)=\{ (\rho,\ve)\in \R^+\times \R^+ : C(\rho+\ve)^2\leq \rho\}.
\eq
We see that if $(\rho,\ve)\in \mathcal S(C)$, then $\rho,\ve>0$ and
\begin{equation}\label{hyp1-bis}
\ve\leq \frac{\sqrt{\rho}}{\sqrt{C}}-\rho.
\end{equation}
Therefore the set $\mathcal S(C)$ is non-empty and bounded. The shape of this set is depicted in Figure~\ref{fig-set}.
\begin{figure}[ht!]
\begin{center}
  \includegraphics{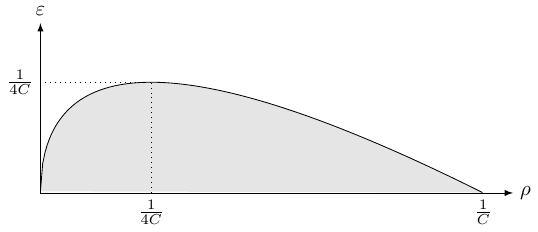}
  \end{center}
  \caption{The shape of the set $\mathcal S(C)$.}
  \label{fig-set}
\end{figure}
In particular, we infer from \eqref{hyp1-bis} that if $(\rho,\ve)\in \mathcal S(C)$, then
\bq\label{cota-rho}
\rho\leq \frac1C \ \text{ and } \  \ve \leq \frac1{4C}.
\eq
In addition, if $\tilde C\geq C$, then
\bq\label{encajonados}\mathcal S(\tilde C)\subseteq \mathcal S(C).\eq
Moreover, taking for instance $\rho=1/(32C)$, Theorem~\ref{thm:cauchy} asserts that for fixed $\alpha\in(0,1]$,
we can take for instance $\ve=1/(32C)$ (that depends  on $\alpha$ and $N$, but not on the $L^{\infty}$-norm of the initial data)
 such that for any given initial condition $u^0\in L^{\infty}(\R^N)$ with $[u^0]_{BMO}\leq \ve$,  there exists a global (smooth) solution $u\in X(\R^N\times \R^{+}; \C)$ of \eqref{DNLS}. Notice that $u^0$ is allowed to have a large $L^\infty$-norm as long as $[u^0]_{BMO}$ is sufficiently small;
this is a weaker requirement that asking for  the $L^\infty$-norm of $u^0$
to be sufficiently small, since
\bq
\label{ineq-BMO}
[f]_{BMO}\leq 2\norm{f}_{L^\infty},\qquad \text{for all} f\in L^\infty(\R^N).
\eq
\begin{remark}
The smallness condition in \eqref{cond-small} is necessary for the uniqueness of the solution.
As we will see in Subsection~\ref{sub-self-2}, at least in dimension one, it is possible to construct multiple solutions
of \eqref{duhamel} in $X(\R^N\times \R^+;\C)$, if $\alpha$ is close enough to 1.
\end{remark}

The aim of this section is to prove Theorem~\ref{thm:cauchy} using a fixed-point technique.
To this pursuit  we write \eqref{duhamel} as
\bq\label{duhamel2}
u(t)=\boT_{u^0}(u)(t),
\eq
where
\bq\label{op-definition}
\boT_{u^0}(u)(t)=S_\alpha(t) u^0+T(g(u))(t) \quad \text{ and }\quad T(f)(t)=\int_0^t S_\alpha(t-s)f(s)\, ds.
\eq
In the next lemmas we study the semigroup $S_\alpha$ and the operator $T$ to establish  that the application $\boT_{u^0}$
is a contraction
on the ball
$$\boB_\rho( u^0)=\{u \in X(\R^N\times \R^+;\C) : \norm{u-S_\alpha(t) u^0}_{X}\leq \rho\},$$
for some $\rho>0$ depending on the size of the initial data.
\begin{lemma}\label{lemma0}
 There exists $C_0>0$ such that for all $f\in BMO^{-1}_{\alpha}(\R^N)$,
\bq
\sup_{t>0}\sqrt{t}\norm{S_\alpha(t) f}_{L^\infty(\R^N)}\leq C_0\norm{f}_{BMO^{-1}_{\alpha}}.
\eq
\end{lemma}
\begin{proof}
The proof in the case $\alpha=1$ is done in \cite[Lemma 16.1]{lemarie}. For $\alpha\in (0,1)$,
decomposing $S_\alpha(t)=S_\alpha(t-s)S_\alpha(s)$ and using  the decay properties of the  kernel $G_\alpha$ associated with the operators $S_{\alpha}(t)$ (see~\eqref{semigrupo}), we can check that the same proof still applies.
\end{proof}
\begin{lemma}\label{lemmaT}
There exists $C_1\geq 1$ such that for all $f\in Y(\R^N\times \R^+;\C)$,
\begin{equation}\label{ineq-lemmaT}
\norm{T(f)}_{X}\leq C_1\norm{f}_Y.
\end{equation}

\end{lemma}
\begin{proof}
Estimate \eqref{ineq-lemmaT}  can be proved  using the arguments given in \cite{koch-tataru} or
\cite{wang}. For the convenience of the reader, we sketch the proof following the lines in \cite[Lemma 3.1]{wang}.
 By scaling and translation, it suffices to show that
\bq\label{aestimar}
\abs{T(f)(0,1)}+\abs{\grad T(f)(0,1)}+\left(\int_{Q_1(0)}\abs{\grad T(f)}^2\right)^{1/2}\lesssim\norm{f}_Y.
\eq
Let $B_r=B_r(0)$. Setting $W=T(f)$, we have
\begin{align*}
W(0,1)=&\int_0^1 \int_{\R^N}G_{\alpha}(-y,1-s)f(y,s)dy ds \\
=&\left(\int_{1/2}^1\int_{\R^N}+\int_{0}^{1/2}\int_{B_2}+\int_{0}^{1/2}\int_{\R^N\setminus B_2} \right)G_{\alpha}(-y,1-s)f(y,s)dy ds\\
:=&I_1+I_2+I_3.
\end{align*}
Since $\abs{G_\alpha(y,1-s)}=\dfrac{ e^{-\frac{\alpha\abs{y}^2}{4(1-s)} } }{ (4\pi (1-s))^{N/2}},$
we obtain
 \begin{align*}
 \abs{I_1}&\leq  \int_{1/2}^1\int_{\R^N}\abs{G_{\alpha}(-y,1-s)}\abs{f(y,s)}dy ds\\
  &\leq \sup_{\frac{1}{2}\leq s\leq 1}\norm{f(s)}_{L^\infty}
    \left(\int_{\frac{1}{2}}^{1}\int_{\mathbb{R}^{n}}\abs{G_\alpha(-y,\ 1-s)}dyds\right)\lesssim \Vert f\Vert_{Y},
 \end{align*}
  \begin{align*}
 \abs{I_2}&\leq  \int_{0}^{1/2}\int_{B_2}\abs{G_{\alpha}(-y,1-s)}\abs{f(y,s)}dy ds\\
  &\lesssim
\sup_{0\leq s\leq\frac{1}{2}}\Vert G_\alpha(\cdot, 1-s)\Vert_{L^{\infty}(\mathbb{R}^{N})}
\int_{B_{2}\times[0,\frac{1}{2}]}|f(y,\ s)|dyds\lesssim  \Vert f\Vert_{Y}
 \end{align*}
and
\begin{align*}
 \abs{I_3}&\leq  \int_{0}^{1/2}\int_{\R^N \setminus B_2}\abs{G_{\alpha}(-y,1-s)}\abs{f(y,s)}dy ds\\
 &\leq\ C\int_{0}^{\frac{1}{2}}\int_{\mathbb{R}^{N}\backslash B_{2}}e^{-\frac{\alpha|y|^{2}}{4}}|f(y,\ s)|dyds\\
&\leq\ C\left(\sum_{k=2}^{\infty}k^{n-1}e^{-\alpha\frac{k^{2}}{4}}\right)
\left(\sup_{y\in \mathbb{R}^{N}}\int_{Q_{1}(y)}|f(y,\ s)|dyds\right)\lesssim \Vert f\Vert_{Y}.
   \end{align*}
The quantity $\abs{\grad T(f)(0,1)}$ can be bounded in a similar way. The last term in the l.h.s.\ of \eqref{aestimar}
can be controlled using an energy estimate. Indeed,  $W$ satisfies the equation
\bq\label{eq-energy}
i\partial_t W+(\beta -i\alpha)\Delta W=if
\eq
with initial condition $W(\cdot,0)=0$. Let $\eta\in C_0^\infty(B_2)$
be a real-valued cut-off function such that $0\leq \eta\leq 1$ on $\R^N$ and $\eta=1$ on $B_1$.
By multiplying \eqref{eq-energy} by $ -i \eta^2 \overline W$, integrating and taking real part, we get
\bqq
\frac12 \partial_t \int_{\R^N}\eta^2 \abs{W}^2+\alpha\int_{\R^N}\eta^2\abs{\grad W}^2
+2\Re\left(
(\alpha+i\beta)\int_{\R^N}\eta\grad \eta \overline W\grad W
\right)
=\int_{\R^N}\eta^2\Re(f \overline W).
\eqq
Using that $\abs{\alpha+i\beta}=1$ and integrating in time between $0$ and $1$, it follows that
$$
\frac12 \int_{\R^N}\eta^2 \abs{W(x,1)}^2 +\alpha\int_{\R^N\times[0,1]}\eta^2\abs{\grad W}^2
\leq
\int_{\R^N\times[0,1]} (2 \eta \abs{\grad \eta} \abs{W}\abs{\grad W}+\eta^2\abs{f}\abs{W}).
$$
From the inequality $ab\leq \ve a^2+b^2/(4\ve),$ with $a=\eta\abs{\grad W}$, $b=2\abs{\grad \eta} \abs{W}$ and $\ve=\alpha/2$,
we deduce that
$$
\frac{\alpha}2 \int_{\R^N\times[0,1]}\eta^2\abs{\grad W}^2 \leq
 \int_{\R^N\times[0,1]}\big(\frac{2}{\alpha}\abs{\grad \eta}^2\abs{W}^2+\eta^2\abs{f}\abs{W}\big).
$$
By the  definition of $\eta$, this implies that
\bq\label{desigualdad}
\norm{\grad W}^2_{L^2(B_1\times[0,1])}\lesssim
\norm{ W}^2_{L^\infty(B_2\times[0,1])}+\norm{ W}_{L^\infty(B_2\times[0,1])}\norm{f}_{L^1(B_2\times[0,1])}.
\eq
From the first part of the proof, we have 
$$\norm{ W}_{L^\infty(B_2\times[0,1])}\leq C \norm{f}_{Y}.$$
Using also that
$$\norm{f}_{L^1(B_2\times[0,1])}\lesssim \norm{f}_{Y},$$
we conclude from \eqref{desigualdad} that
$$\norm{\grad W}_{L^2(B_1\times[0,1])}\lesssim \norm{f}_{Y},$$
which finishes the proof.
\end{proof}
\begin{lemma}\label{lemma-bola}
Let $\alpha\in (0,1]$ and $\rho, \varepsilon, L>0$. There exists $C_2\geq 1$, depending on $\alpha$ and $N$, such that for all
 $u^0\in L^\infty(\R^N)$
\bq\label{semigroup}
\norm{S_\alpha(t) u^0}_X\leq C_2 (\norm{u^0}_{L^\infty}+[u^0]_{BMO}).
\eq
If in addition $\norm{u^0}_{L^\infty}\leq L$ and  $[u^0]_{BMO}\leq \ve$,
then for all $u\in \boB_\rho(u^0)$ we have
\bq\label{est-bola}
\sup_{t>0} \norm{u}_{L^\infty}\leq C_2( \rho+L )\quad \text{and}\quad [u]_{X}\leq C_2(\rho+\ve).
\eq
\end{lemma}
\begin{proof}
We first control $\norm{S_\alpha(t) u^0}_X$. On the one hand, using the definition of $G_\alpha$ and the relation $\alpha^2+\beta^2=1$, we obtain
$$
\norm{S_\alpha(t)u^0 }_{L^\infty}=
\norm{G_\alpha\ast u^0}_{L^\infty}\leq \norm{G_\alpha}_{L^1} \norm{u^0}_{L^\infty}=
\alpha^{-\frac{N}{2}} \norm{u^0}_{L^\infty}, \qquad \forall\, t>0.
$$
Thus
\begin{equation}
 \label{lemma-bola-1}
 \sup_{t>0} \norm{S_\alpha(t)u^0}_{L^\infty}\leq \alpha^{-\frac{N}{2}} \norm{u^0}_{L^\infty}.
\end{equation}
On the other hand, using Lemma~\ref{lemma0}, Theorem~\ref{tataru} and \eqref{equiv-BMO},
\begin{equation} \label{lemma-bola-2}
\begin{aligned}
\phantom{.}[ S_\alpha(t) u^0 ]_{X}&=
\sup_{t>0}\sqrt{t}\norm{\grad S_\alpha(t)u^0}_{L^\infty}+
\sup_{\substack{ x\in \R^N\\ r>0}}\left( \frac1{r^N}\int_{Q_r(x)} \abs{\grad S_\alpha(t) u^0 }^2\,dt\,dy \right)^\frac12
 \\
&\lesssim \norm{\grad u^0}_{BMO_\alpha^{-1}}+[u^0]_{BMO_\alpha}\\
&\lesssim  [u^0]_{BMO_\alpha} \\
&\lesssim  [u^0]_{BMO}.
\end{aligned}
\end{equation}
The estimate in \eqref{semigroup} follows from \eqref{lemma-bola-1} and \eqref{lemma-bola-2},
and we w.l.o.g.\ can choose $C_2\geq 1$.

Finally, using \eqref{lemma-bola-1}, given $u^0$ such that $\norm{u^0}_{L^\infty}\leq L$ and $[u^0]_{BMO}\leq \varepsilon$, for all $u\in \boB_\rho( u^0)$   we have
$$
\norm{u}_{L^\infty}
\leq \norm{u-S_\alpha(t)u^0}_{L^\infty}+\norm{S_\alpha(t)u^0}_{L^\infty}
\leq \norm{u-S_\alpha(t)u^0}_{X}+\norm{S_\alpha(t)u^0}_{L^\infty}
\leq C_2( \rho+L),
$$
and, using  \eqref{lemma-bola-2},
$$
[u]_{X}\leq [u-S_\alpha(t)u^0]_{X}+ [S_\alpha(t)u^0]_{X}
\leq \norm{u-S_\alpha(t)u^0}_{X}+[S_\alpha(t)u^0]_{X}
\leq C_2(\rho+\varepsilon),
$$
which finishes the proof of \eqref{est-bola}.
\end{proof}
Now we proceed to bound the nonlinear term
$$g(u)=-2i(\beta-i\alpha)\frac{\bar u (\grad u)^2}{1+\abs{u}^2}.$$
\begin{lemma}\label{lemma-g} For all $u\in X(\R^N\times \R^+;\C)$, we have
 $$\qquad\norm{g(u)}_Y\leq  [u]_X^2.$$
\end{lemma}
\begin{proof}
Let $u\in X(\R^N\times \R^+;\C)$. Using \eqref{est-g} and
the definitions of the norms in  $Y$ and $X$, it follows that
\begin{align*}
\norm{g(u)}_{Y}\leq
\left( \sup_{t>0}\sqrt{t}\norm{\grad u}_{L^\infty}\right)^2
+
\sup_{\substack{ x\in \R^N\\ r>0}}\frac1{r^N}\int_{Q_r(x)} \abs{\grad u }^2\,dt\,dy
\leq    [u]_X^2.
\end{align*}
\end{proof}
Now we have all the estimates to prove that $\boT_{u^0}$ is a  contraction on $\boB_\rho(u^0)$.
\begin{prop}\label{T-contraction}
Let $\alpha\in(0,1]$ and $\rho, \ve>0$. Given any $u^0\in L^{\infty}(\R^N)$ with $[u^0]_{BMO}\leq \ve$, 
the operator  $\boT_{u^0}$ given in \eqref{op-definition} defines a contraction on $\boB_\rho(u^0)$, whenever $\rho$ and $\ve$ satisfy
\bq\label{condition-rho}
8C_1 C_2^2(\rho+\ve)^2\leq \rho.
\eq
Moreover, for all $u,v\in X(\R^N\times \R^+;\C)$,
\bq\label{contraction2}
 \norm{ T(g(u))-T(g(v))}_{X}\leq C_1(2[u]^2_X+[u]_X+[v]_X)
 \norm{u-v}_{X}.
 \eq
Here, $C_1\geq1$ and $C_2\geq 1$ are the constants in Lemmas~\ref{lemmaT} and \ref{lemma-bola}, respectively.
\end{prop}

\begin{remark}
Using the notation introduced in \eqref{set-S}, the hypothesis \eqref{condition-rho} means that
$(\rho,\ve)\in\mathcal S(8C_1C_2^2)$. Therefore, by  \eqref{cota-rho},
\bq\label{est-rho}
\rho\leq \frac1{8C_1C_2^2}, \ \text{ and } \ \ve \leq \frac1{32C_1C_2^2},
\eq
so $\rho$ and $\ve$ are actually small. Since $C_1,C_2\geq 1$, we have
\bq\label{para-contraction}
C_2(\rho+\ve)\leq \frac{5}{32}.
\eq
\end{remark}
\begin{proof}
Let $u^0\in L^\infty(\R^N)$ with $\norm{u^0}_{L^\infty}\leq L$ and  $[u^0]_{BMO}\leq \ve$, and  $u\in \boB_\rho(u^0)$.
Using Lemma~\ref{lemmaT}, Lemma~\ref{lemma-bola} and Lemma~\ref{lemma-g}, we have
\bqq
\norm{\boT_{u^0}(u)-S_\alpha(t)u^0}_X=\norm{T(g(u))}_X\leq C_1 \norm{g(u)}_Y\leq C_1 [u]_{X}^{2}
\leq C_1 C_2^2 (\rho + \ve)^2.
\eqq
Therefore  $\boT_{u^0}$ maps $\boB_\rho(u^0)$ into itself provided that
\bq\label{proof-bola}
C_1C_2^2(\rho+\ve)^2\leq \rho.
\eq
Notice that by \eqref{encajonados}, the condition \eqref{condition-rho} implies that \eqref{proof-bola} is satisfied.

To prove \eqref{contraction2}, we use the decomposition
\bqq
g(u)-g(v)=-2i(\beta-i\alpha)\left[
\left(\frac{\bar u}{1+\abs{u}^2}-\frac{\bar v}{1+\abs{v}^2}\right)
(\grad u)^2
+
\frac{\bar v}{1+\abs{v}^2}
((\grad u)^2-(\grad v)^2))
\right].
\eqq
Since
\bqq
 \left| \frac{\bar u}{1+\abs{u}^2}-\frac{\bar v}{1+\abs{v}^2}\right|\leq
|u-v|\frac{1+|u|\, |v|}{(1+|u|^2)(1+|v|^2)}
\leq |u-v|,
\eqq
and using  \eqref{ineq0},
we obtain
$$
|g(u)-g(v)|
\leq 2\, |u-v|\, |\grad u|^2 +|\grad u - \grad v|\, (|\grad u|+|\grad v|).
$$
Therefore
\bq\label{diff-T}
\norm{g(u)-g(v)}_Y\leq 2\norm{\abs{u-v}\abs{\grad u}^2}_Y+\norm{\abs{\grad u-\grad v}(\abs{\grad u}+\abs{\grad v})}_Y:=I_1+I_2.
\eq
For $I_1$, it is immediate that
\bq\label{I1}
 I_1 \leq 2 \sup_{t>0}\norm{u-v}_{L^\infty}
 \left[
  \Big(\sup_{t>0} \sqrt{t} \norm{\grad u}_{L^\infty}\Big)^2+
  \sup_{\substack{ x\in \R^N\\ r>0}} \frac1{r^N}\int_{Q_r(x)} \abs{\grad u}^2\,dt\,dy
  \right]
 \leq 2\norm{u-v}_X [u]^2_X.
\eq
Similarly, using the Cauchy--Schwarz inequality,
\bq\label{I2}
\begin{aligned}
 I_2\leq&
 \left(\sup_{t>0}\sqrt{t}\norm{\grad u-\grad v}_{L^\infty}\right)
 \left(\sup_{t>0}\sqrt{t}(\norm{\grad u}_{L^\infty}+\norm{\grad v}_{L^\infty})\right)\\
&+\sup_{\substack{ x\in \R^N\\ r>0}}\frac1{r^{N}}\left(\norm{\grad u-\grad v}_{L^2(Q_r(x))}\right)
\left(\norm{\grad u}_{L^2(Q_r(x))}+\norm{\grad v}_{L^2(Q_r(x))}\right)
\\
\leq &\norm{u-v}_X([u]_X+[v]_X).
\end{aligned}
\eq
Using Lemma~\ref{lemmaT}, \eqref{diff-T}, \eqref{I1} and \eqref{I2}, we conclude that
\bq\label{contraction-proof}
\norm{ T(g(u))-T(g(v))}_{X}\leq C_1 (2[u]^2_X+[u]_X+[v]_X)\norm{u-v}_{X}.
\eq
Let  $u,v\in \boB_\rho(u^0)$, by Lemma~\ref{lemma-bola} and \eqref{para-contraction}
\bq\label{cota-diff1}
[u]_X\leq C_2(\rho+\ve)\leq \frac5{32},
\eq
so that
\bq\label{cota-diff2}
2[u]^2_X+[u]_X+[v]_X\leq \frac{37}{16}C_2(\rho+\ve)<3C_2(\rho+\ve).
\eq
Then \eqref{contraction-proof}  implies that
\bq\label{proof-contraction}
\norm{ \boT_{u^0}(u)-\boT_{u^0}(v)}_{X}\leq 3C_1C_2(\rho+\ve)\norm{u-v}_X.
\eq
From \eqref{est-rho}, we conclude that
\bq\label{cota-diff3}3C_1C_2(\rho+\ve)\leq \frac{15}{32}\leq \frac12,\eq
and then \eqref{proof-contraction} yields that the operator
$\boT_{u^0}$ defined in \eqref{op-definition} is a contraction on $\boB_\rho(u^0)$.
This concludes the proof of the proposition.

\end{proof}
%
%
\begin{proof}[Proof of Theorem~\ref{thm:cauchy}]
Let us set $C=C_1C_2^2$ and $K=C_2$,  where $C_1$ and $C_2$ are the constants in Lemma~\ref{lemmaT} and Lemma~\ref{lemma-bola} respectively.
Since $\rho$ satisfies \eqref{hyp1},  Proposition~\ref{T-contraction} implies that there exists a
solution $u$ of equation \eqref{duhamel2} in the ball $\boB_\rho(u^0)$, and in particular from Lemma~\ref{lemma-bola}
$$
\sup_{t>0}\norm{u}_{L^\infty}\leq K(\rho+L)
\qquad \text{and} \qquad
 [u]_X\leq K(\rho+\ve).
$$
To prove the uniqueness part of the theorem,
let us assume that $u$ and $v$ are solutions of \eqref{duhamel} in $X(\R^N\times \R^{+}; \C)$  such that
\begin{equation}\label{two-con}
  [u]_{X},[v]_{X}\leq K(\rho+\ve),
\end{equation}
with the same initial condition $u^0$. By the definitions of $C$ and $K$, \eqref{hyp1} and \eqref{two-con},
the estimates in \eqref{est-rho} and \eqref{para-contraction} hold.
It follows that \eqref{cota-diff1}, \eqref{cota-diff2} and \eqref{cota-diff3} are satisfied.
Then, using \eqref{contraction2},
\begin{align*}
\norm{u-v}_{X}
&=\norm{T(g(u))-T(g(v))}_X\leq C_1  (2[u]^2_X+[u]_X+[v]_X) \norm{u-v}_{X} \\
&\leq   \frac{1}{2} \norm{u-v}_{X}.
\end{align*}
From which it follows that $u=v$.

To prove the dependence of the solution with respect to the initial data (part {\it{(iv)}}),
consider $u$ and $v$ solutions of \eqref{duhamel} satisfying \eqref{two-con}
with initial conditions $u^0$ and $v^0$.
Then, by definition, $u=\boT_{u^0}(u)$, $v=\boT_{v^0}(v)$ and
\bqq
\norm{u-v}_X=\norm{\boT_{u^0}(u)-\boT_{v^0}(v)}_X\leq \norm{S_\alpha(u^0-v^0)}_{X}+
 \norm{T(g(u))-T(g(v))}_{X}.
\eqq
Using \eqref{ineq-BMO}, \eqref{semigroup} and \eqref{contraction2} and arguing as above, we have
\begin{align*}
\norm{u-v}_X
&\leq  C_2(\norm{u^0-v^0}_{L^\infty}+ [u^0-v^0]_{BMO})+C_1 (2[u]_X^2+[u]_X+[v]_X)\norm{u-v}_X\\
&\leq 3C_2 \norm{u^0-v^0}_{L^\infty}+\frac{1}{2}\norm{u-v}_X.
\end{align*}
This yields \eqref{cont-CI}, since $K=C_2$.

The assertions in {\em(ii)} and {\em(iii)}  follow from Theorem~\ref{thm-equiv}.
\end{proof}
%
\subsection{The Cauchy problem for the LLG equation}
\label{sec-LLG}
By using the inverse of the stereographic projection
$\P^{-1} :\C\to \S^2\setminus\{0,0,-1\}$, that is explicitly given by
$\m=(m_1, m_2, m_3)=\P^{-1}(u)$, with
\bq\label{inverse-P}
m_1=\frac{2\Re u}{1+\abs{u}^2}, \quad m_2=\frac{2\Im u}{1+\abs{u}^2}, \quad m_3=\frac{1-\abs{u}^2}{1+\abs{u}^2},
\eq
we will be able to  establish the following
  global well-posedness result for \eqref{LLG}.
\begin{thm}\label{thm:cauchy-LLG}
Let $\alpha\in(0,1]$. There exist constants $C\geq 1$ and $K\geq 4$, such that for any $\delta\in (0,2]$,  $\ve_0>0$ and  $\rho>0$
such that
\bq\label{cond-LLG}
 8K^4C \delta^{-4}(\rho+8\delta^{-2}\ve_0)^2\leq \rho,
\eq
if  $\m^0=(m_1^0,m_2^0,m_3^0)\in L^\infty(\R^N;\S^2)$ satisfies
\bq\label{cond-CI-1}
\inf_{\R^N}{m^0_3}\geq-1+\delta \quad\textup{ and }\quad[\m^0]_{BMO}\leq \ve_0,
\eq
then there exists a unique solution $\m=(m_1,m_2,m_3)\in X(\R^N\times \R^+;\S^2)$ of \eqref{LLG}
such that
\bq\label{cond-small-m}
 \inf_{\substack{x\in \R^N\\ t>0}} m_3(x,t)\geq -1+\frac{2}{1+K^2(\rho+\delta^{-1})^2}
 \quad \textup{ and }\quad [\m]_{X}\leq 4K(\rho+8\delta^{-2}\ve_0).
\eq
 Moreover, we have the following properties.
\begin{itemize}
\item[{\it{i)}}] $\m\in C^{\infty}(\R^N\times \R^{+}; \mathbb{S}^2)$.
\item[{\it{ii)}}] $|\m(\cdot, t)-\m^{0}|\longrightarrow 0$ in $S'(\R^N)$ as $t\longrightarrow 0^{+}$.
\item[{\it{iii)}}] Assume that $\m$ and $\n$ are respectively smooth solutions to \eqref{duhamel} satisfying \eqref{cond-small-m} with initial conditions $\m^0$ and $\n^0$
 satisfying \eqref{cond-CI-1}. Then
\bq\label{cont-CI-1}
 \norm{\m-\n}_{X}\leq 120K\delta^{-2}\norm{\m^0-\n^0}_{L^\infty}.
\eq
\end{itemize}
\end{thm}
\begin{remark}\label{rem-set-delta}
 The restriction \eqref{cond-LLG} on the parameters is similar to \eqref{condition-rho}, but we need to include $\delta$.
To better understand the role of $\delta$, we can proceed as before. Indeed, setting for $a,\delta>0$,
$$
 \mathcal S_\delta(a)=\{ (\rho,\ve_0)\in \mathbb{R}^{+}\times \mathbb{R}^{+} : a\delta^{-4}(\rho+8\delta^{-2}\ve_0)^2\leq \rho \},$$
we see that its shape is similar to the one in Figure~\ref{fig-set}. It is simple to verify that
for any $(\rho,\ve_0)\in \mathcal S_\delta(a)$, we have the bounds
\bq\label{cotas-rho}
\rho\leq \frac{\delta^4}{a}\quad\text{and}\quad \ve_0\leq \frac{\delta^6}{32a},
\eq
and the maximum value $\ve_0^*=\frac{\delta^6}{32a}$ is attained at $\rho^*=\frac{\delta^4}{4a}$. Also, the sets are well ordered, i.e.\
 if $\tilde a\geq a>0$, then $\mathcal S_\delta(\tilde a)\subseteq\mathcal S_\delta(a)$.
\end{remark}

We emphasize that  the first condition in \eqref{cond-CI-1} is rather technical.
Indeed, we need the essential range of $\m^0$ to be far from the South Pole in order to use
the stereographic projection. In the case $\alpha=1$, Wang~\cite{wang} proved
the global well-posedness using only the second restriction in \eqref{cond-CI-1}.
It is an open problem to determinate if this condition is necessary in the case $\alpha\in(0,1)$.

The choice of the South Pole is of course arbitrary. By using the invariance of \eqref{LLG} under rotations,
we have the existence of solutions provided that the essential range of the initial condition  $\m^0$ is far from an arbitrary point $\QQ\in \S^2$. Precisely,

\begin{cor}\label{cor:cauchy-LLG}
Let  $\alpha\in(0,1]$,  $\QQ\in \S^2$, $\delta\in (0,2]$,  and  $\ve_0,\rho>0$ such that
\eqref{cond-LLG} holds. Given  $\m^0=(m_1^0,m_2^0,m_3^0)\in L^\infty(\R^N;\S^2)$ satisfying
\bqq
\inf_{\R^N}\abs{\m^0-\QQ}^2\geq 2\delta \quad\textup{ and }\quad[\m^0]_{BMO}\leq \ve_0,
\eqq
there exists a unique smooth solution $\m\in X(\R^N\times \R^+;\S^2)$ of \eqref{LLG}
with initial condition $\m^0$ such that
 \bq\label{cond-cor}
 \inf_{\substack{x\in \R^N\\ t>0}} \abs{\m(x,t)-\QQ}^2\geq \frac{4}{1+K^2(\rho+\delta^{-1})^2}
 \quad \textup{ and }\quad [\m]_{X}\leq 4K(\rho+8\delta^{-2}\ve_0).
 \eq
 \end{cor}
For the sake of clarity, before proving Theorem \ref{thm:cauchy-LLG},
we provide a precise meaning of what we refer to as a weak and smooth
global solution of the \eqref{LLG} equation. The definition below is motivated by  the following vector identities for a smooth function
$\m$  with $\abs{\m}=1$:
\begin{gather*}
 \m\times \Delta \m=\div(\m\times \grad \m),\\
 -\m\times (\m\times \Delta \m)=\Delta \m+\abs{\grad \m}^2\m.
\end{gather*}

\begin{definition}
Let $T\in (0,\infty]$ and $\m^0\in L^\infty(\R^N;\S^2)$.
We say that  $$\m\in L^\infty_{\loc}((0,T),H^1_{\loc}(\R^N;\S^2))$$ is a weak solution of \eqref{LLG} in $(0,T)$
with initial condition $\m^0$  if
\bqq
-\langle \m,\partial_t \varphi\rangle=\beta \langle \m\times \grad \m,\grad \varphi\rangle
-\alpha \langle \grad \m,\grad \varphi\rangle +\alpha \langle \abs{\grad \m}^2\m, \varphi\rangle,
\eqq
and
\bq\label{def-CI}
\norm{ (\m(t)-\m^0)\vp}_{L^1}\to 0, \quad \textup{ as }t\to 0^+,
\qquad \text{for all }
\varphi\in C_0^\infty(\R^N\times (0,T)).
\eq
If $T=\infty$,  and in addition  $\m\in C^\infty(\R^N\times \R^+)$,
we say that $\m$ is a smooth global solution of \eqref{LLG} in $\R^N\times \R^+$
with initial condition $\m^0$. Here $\langle \cdot, \cdot\rangle$ stands for
$$
\langle \f_1, \f_2\rangle=\int_{0}^\infty\int_{\R^N} \f_1\cdot\f_2 \, dx\,dt.
$$
\end{definition}
With this definition, we see the following: Assume that $\m$ is  a smooth global solution of \eqref{LLG} with initial condition $\m^0$ and consider its stereographic projection $\P(\m)$. If $\P(\m)$ and $\P(\m^0)$ are well-defined, 
then $\P(\m)\in C^{\infty}(\R^N\times \R^{+}; \C)$ satisfies \eqref{DNLS} pointwise, and
$$
  \lim_{t\rightarrow 0^{+}} \P(\m)=\P(\m^0) \qquad \text{in } S'(\R^N).
$$
Therefore, if in addition  $\P(\m)\in X(\R^N\times \R;\C)$, then $\P(\m)$ is a smooth
global solution of \eqref{DNLS} with initial condition $\P(\m^0)$. Reciprocally,
suppose that  $u\in  X(\R^N\times \R^+;\C)\cap C^\infty(\R^N\times \R^+)$
is a solution of \eqref{duhamel} with initial condition $u^0\in L^\infty(\R^N)$ such that
\eqref{conver-en-0} holds. If  $\P^{-1}(u)$  and $\P^{-1}(u^0)$ are  in  appropriate spaces,
then $\P^{-1}(u)$ is a global smooth solution of \eqref{LLG} with initial condition $\P^{-1}(u^0)$.
The above (formal) argument allows us to obtain Theorem~\ref{thm:cauchy-LLG}
from Theorem~\ref{thm:cauchy} once we have established good estimates for the mappings
$\P$ and $\P^{-1}$. In this context, we have the following
\begin{lemma}\label{lem-projection} Let $u,v\in C^1(\R^N;\C)$,
$\m=(m_1,m_2,m_3),\n=(n_1,n_2,n_3)\in C^1(\R^N;\S^2)$.
\begin{enumerate}
\item[a)] Assume that  $\inf\limits_{\R^N} m_3\geq -1+\delta$ and
$\inf\limits_{\R^N} n_3\geq -1+\delta$
for some constant $\delta\in (0,2]$. If  $u=\P(\m)$ and  $v=\P(\n)$, then
\begin{align}
\label{des-u-v}\abs{u(x)-v(x)} &\leq \frac4{\delta^2}\abs{\m(x)-\n(x)},\\
\label{des-BMO}[u]_{BMO}&\leq \frac8{\delta^2}[\m]_{BMO},\\
\label{des-grad-u} \abs{\grad u(x)}&\leq \frac{4}{\delta^2}\abs{\grad \m(x)},
\end{align}
for all $x\in \R^N$.
\item[b)] Assume that $\norm{u}_{L^\infty}\leq M$, $\norm{v}_{L^\infty}\leq M$, for some constant $M\geq 0$. If $\m=\P^{-1}(u)$ and $\n=\P^{-1}(v)$ , then
\begin{align}
\label{des-m3}\inf_{\R^N} m_3&\geq -1+\frac{2}{1+M^2},\\
\label{des-m-n} \abs{ \m(x)-  \n(x)} &\leq 3\abs{u(x)-v(x)},\\
\label{des-grad} \abs{\grad{\m(x)}}&\leq 4\abs{\grad u(x)},\\
\label{des-grad-m-n} \abs{\grad{\m(x)- \grad \n(x)}}&\leq
4\abs{\grad u(x)-\grad v(x)}+12\abs{u(x)-v(x)}
(\abs{\grad u(x)}+\abs{\grad v(x)}).
\end{align}
\end{enumerate}
\end{lemma}
 \begin{proof}
 In the proof we will use the notation $\check m:=m_1+im_2$.
To establish \eqref{des-u-v}, we write
\bqq
u(x)-v(x)=\frac{\check m (x)-\check n(x)}{1+m_3(x)}+\frac{\check n(x)(n_3(x)-m_3(x))}{(1+m_3(x))(1+n_3(x))}.
\eqq
Hence, since $\abs{\check n}\leq 1$, $m_3(x)+1\geq \delta$ and $n_3(x)+1\geq \delta$, $\forall x\in\mathbb{R}^N$,
\bqq
\abs{u(x)-v(x)}\leq \frac{\abs{\check m (x)-\check n(x)}}{\delta}+
\frac{\abs{n_3(x)-m_3(x)}}{\delta^2}.
\eqq
Using that
\bq\label{des-check}
\abs{\check m-\check n}\leq \abs{\m-\n}\eq
and that
$$\max\left\{ \frac1{a},\frac1{a^2}\right\}\leq \frac2{a^2}, \textup{ for all }a\in (0,2],$$
we obtain \eqref{des-u-v}. The same argument also shows that
\bq\label{u-y}
\abs{u(y)-u(z)} \leq \frac4{\delta^2}\abs{\m(y)-\m(z)},\quad \textup{ for all }y,z\in \R^N.
\eq
To verify \eqref{des-BMO}, we recall the following inequalities in BMO (see \cite{brezis-nir}):
\bq\label{BMO-brezis}
[f]_{BMO}\leq \sup_{x\in\R^N}\fint_{B_r(x)}\fint_{B_r(x)}\abs{f(y)-f(z)}\,dy\,dz\leq 2[f]_{BMO}.
\eq
Estimate \eqref{des-BMO} is an immediate consequence of this inequality and \eqref{u-y}.
To prove \eqref{des-grad-u} it is enough to remark that
$$
\abs{\grad u}\leq \frac{2}{\delta^2}(\abs{\grad m_1}+\abs{\grad m_2}+\abs{\grad m_3})\leq \frac{4}{\delta^2}\abs{\grad \m}.
$$

We turn into {\it (b)}.
Using the explicit formula for $\P^{-1}$ in \eqref{inverse-P}, we can write
$$m_3=-1+\frac{2}{1+\abs{u}^2}.$$
Since $\norm{u}_{L^\infty}\leq M$, we obtain \eqref{des-m3}.

To show \eqref{des-m-n}, we compute
\begin{align}
\label{check-m-n}\check m-\check n&=
\frac{2 u}{1+\abs{u}^2}-\frac{2 v}{1+\abs{v}^2}=\frac{2(u-v)+2uv(\bar v-\bar u)}{(1+\abs{u}^2)(1+\abs{v}^2)},\\
\label{check-m3-n3}m_3- n_3&=\frac{1-\abs{u}^2}{1+\abs{u}^2}-\frac{1-\abs{v}^2}{1+\abs{v}^2}=
\frac{2(\abs{v}^2-\abs{u}^2)}{(1+\abs{u}^2)(1+\abs{v}^2)}.
\end{align}
Using the inequalities
\begin{equation}
\label{ineq01}
 \frac{a}{1+a^2}\leq \frac{1}{2}, \ \
 \frac{1+ab}{(1+a^2)(1+b^2)}\leq 1,
 \ \
 \text{ and }\ \
 \frac{a+b}{(1+a^2)(1+b^2)}\leq 1,\ \
 \text{ for all } a,b\geq 0,
\end{equation}
from \eqref{check-m-n} and \eqref{check-m3-n3} we deduce that
\begin{align}\label{diff-m}
\abs{\check m-\check n}\leq 2\abs{u-v}\quad \text{ and } \quad \abs{m_3- n_3}&\leq 2\abs{u-v}.
\end{align}
Hence
\bqq
\abs{\m-\n}=\sqrt{\abs{\check m-\check n}^2+\abs{m_3- n_3}^2}\leq \sqrt{8}\abs{u-v}\leq 3\abs{u-v}.
\eqq
To estimate the gradient, we compute
\bq
\grad \check m=\frac{2\grad u}{1+\abs{u}^2}-\frac{4u \Re(\bar u \grad u)}{(1+\abs{u}^2)^2},
\eq
from which it follows that
$$
 \abs{\grad \check m}\leq \abs{\grad u}\left(\frac{2}{1+|u|^2}+\frac{4\abs{u}^2}{(1+\abs{u}^2)^2}\right)\leq 3\abs{\grad u},
$$
since $\frac{4a}{(1+a)^2}\leq 1,$ for all $a\geq 0$.
For $m_3$, we have
$$
\grad m_3=-\frac{2 \Re(\bar u \grad u)}{1+\abs{u}^2}-\frac{2 \Re(\bar u \grad u)(1-\abs{u}^2)}{(1+\abs{u}^2)^2}=-\frac{4\Re(\bar u \grad u)}{(1+\abs{u}^2)^2},
$$
and therefore  $\abs{\grad m_3}\leq 2\abs{\grad u}$, since
\bq\label{ineq2}
\frac{a}{(1+a^2)^2}\leq \frac12, \quad \text{for all } a\geq 0.
\eq
Hence
$$
\abs{\grad \m}= \sqrt{\abs{\grad m_1}^2+\abs{\grad m_2}^2+\abs{\grad m_3}^2}\leq \sqrt{13}\abs{\grad u}\leq 4 \abs{\grad u},
$$
which gives \eqref{des-grad}.

In order to prove \eqref{des-grad-m-n}, we start differentiating \eqref{check-m-n}
\bqq
\begin{aligned}
\grad\check  m-\grad\check n=&
2\frac{\grad (u- v)+\grad(uv)(\bar v-\bar u)+uv\grad(\bar v- \bar u)}{(1+\abs{u}^2)(1+\abs{v}^2)}\\
&-4\frac{((u-v)+uv(\bar v-\bar u))(\Re(\bar u\grad u)(1+\abs{v}^2)+\Re(\bar v\grad v)(1+\abs{u}^2))}{(1+\abs{u}^2)^2(1+\abs{v}^2)^2},
\end{aligned}
\eqq
Hence, setting $R=\max\{ \abs{\grad u(x)}, \abs{\grad v(x)} \}$,
\bqq
\begin{aligned}
\abs{\grad\check  m-\grad\check n}\leq &
2\abs{\grad u-\grad v}\left(\frac{1+\abs{u}\abs{v}}{(1+\abs{u}^2)(1+\abs{v}^2)} \right)
+2R\abs{u-v}
\left(\frac{\abs{u}+\abs{v}}{(1+\abs{u}^2)(1+\abs{v}^2)} \right)\\
&+4R\abs{u-v}\left(
\frac{\abs{u}(1+\abs{u}\abs{v})}{(1+\abs{u}^2)^2(1+\abs{v}^2)}+
\frac{\abs{v}(1+\abs{u}\abs{v})}{(1+\abs{u}^2)(1+\abs{v}^2)^2}
\right).
\end{aligned}
\eqq
Using again \eqref{ineq01}, we get
\bqq
\frac{\abs{u}(1+\abs{u}\abs{v})}{(1+\abs{u}^2)^2(1+\abs{v}^2)}\leq
 \frac{\abs{u}}{(1+\abs{u}^2)}
\leq \frac12.
\eqq
By symmetry, the same estimate holds interchanging   $u$ by $v$. Therefore,
invoking again \eqref{ineq01}, we obtain
\bq\label{des-grad-1}
\abs{\grad\check  m-\grad\check n}\leq 2\abs{\grad u-\grad v}+6R\abs{u-v}.
\eq
Similarly, writing $\abs{u}^2-\abs{v}^2=(u-v)\bar u+(\bar u-\bar v)v$,  from \eqref{check-m3-n3}
we have
\bq\label{des-grad-2}
\abs{\grad m_3-\grad n_3}\leq 2\abs{\grad u-\grad v}+6R\abs{u-v}.
\eq
Therefore, since
\bqq
\sqrt{a^2+b^2}\leq a+b, \qquad \forall \, a, b\geq 0,
\eqq
  inequalities \eqref{des-grad-1} and \eqref{des-grad-2} yield \eqref{des-grad-m-n}.
 \end{proof}
Now we have all the elements to establish Theorem~\ref{thm:cauchy-LLG}.

\begin{proof}[Proof of Theorem~\ref{thm:cauchy-LLG}]
We continue to use the constants $C$ and $K$ defined in Theorem~\ref{thm:cauchy}.
We recall that they are given by $C=C_1C_2^2$ and $K=C_2$, where $C_1\geq 1$ and $C_2\geq 1$
are the constants in Lemmas~\ref{lemmaT} and \ref{lemma-bola}, respectively.
In addition, w.l.o.g.\ we assume that
\bq\label{hyp-K}
K=C_2\geq 4,
\eq
in order to simplify our computations.

First we notice that by Remark~\ref{rem-set-delta}, any $\rho$ and $\ve_0$
fulfilling the condition \eqref{cond-LLG}, also satisfy
\bq\label{proof-cond}
8C(\rho+8\delta^{-2}\ve_0)^2\leq \rho,
\eq
since $\delta^4/K^4\leq 1$ (notice that $K\geq 4$ and $\delta\in (0,2]$).

Let $\m^0$ as in the statement of the theorem and set $u^0=\P(\m^0)$.
Using \eqref{des-BMO} in Lemma~\ref{lem-projection}, we have
\bqq
\norm{u^0}_{L^\infty} \leq \Big\Vert{\frac{1}{1+m^0_3}}\Big\Vert_{L^\infty}\leq \frac{1}{\delta} \quad \text{ and } \quad [u^0]_{BMO}\leq \frac{8\ve_0}{\delta^2}.
\eqq
Therefore, bearing in mind \eqref{proof-cond}, we can  apply Theorem~\ref{thm:cauchy} with
$$L:=\frac1{\delta}
\quad \textup{ and } \quad\ve:=8\delta^{-2}\ve_0,
$$
to obtain a smooth solution $u\in X(\R^N\times \R^+;\C)$ to \eqref{duhamel} with initial condition $u^0$.
In particular $u$ satisfies
\bq\label{cota-u}
\sup_{t>0}\norm{u}_{L^\infty}\leq K(\rho+\delta^{-1}) \quad \text{ and }
\quad [u]_{X}\leq K(\rho+8\delta^{-2}\ve_0).
\eq
Defining $\m=\P^{-1}(u)$, we infer that $\m$ is a smooth solution to \eqref{LLG} and, using the fact that  $\| (u(\cdot, t)-u^{0})\varphi  \|_{L^1}\rightarrow 0$ (see \eqref{conver-en-0}) and \eqref{des-m-n},
$$
 |\m(\cdot, t)-\m^{0}|\longrightarrow 0 \qquad \text{in } \mathcal{S}'(\mathbb{R}^N), \text{ as } t\rightarrow 0^{+}.
$$
Notice also that applying Lemma~\ref{lem-projection} we obtain
 \bqq
 \inf_{\substack{x\in \R^N\\ t>0}} m_3(x,t)\geq -1+\frac{2}{1+K^2(\rho+\delta^{-1})^2}
 \quad \textup{ and }\quad [\m]_{X}\leq 4 [u]_{X}\leq  4K(\rho+8\delta^{-2}\ve_0),
 \eqq
which yields \eqref{cond-small-m}.

Let us now prove the uniqueness. Let $\n$ be a another smooth solution of \eqref{LLG} with initial condition $u^0$
satisfying
\bq\label{n-small}
 \inf_{\substack{x\in \R^N\\ t>0}} n_3(x,t)\geq -1+\frac{2}{1+K^2(\rho+\delta^{-1})^2}
  \quad \textup{ and }\quad
  [\n]_{X}\leq  4K(\rho+8\delta^{-2}\ve_0),
\eq
and  let $v=\P(\n)$ be its stereographic projection. Then by \eqref{des-grad-u},
\bq\label{semi-v}
[v]_X\leq \Big(1+K^2(\rho+{\delta^{-1}})^2\Big)^2[\n]_X.
\eq
We continue to control the upper bounds for $[v]_X$ and $[u]_X$ in terms of $\delta$ and the constants $C_1\geq 1$ and $C_2\geq 4$. Notice that since $\rho$ and $\ve_0$ satisfy \eqref{cond-LLG}, from \eqref{cotas-rho} with $a=8K^4C$, it follows that
$$
 \rho\leq \frac{\delta^4}{8K^4C}\qquad \text{and}\qquad \ve_0\leq \frac{\delta^6}{2^8 K^4 C},
$$
or equivalently (recall that $K=C_2$ and $C=C_1C_2^2$)
\begin{equation}
\label{cotas-rho-ve}
 \rho\leq \frac{\delta^4}{8C_1C^6_2}\qquad \text{and}\qquad \frac{8\ve_0}{\delta^2}\leq \frac{\delta^4}{32 C_1 C_2^6}.
\end{equation}
Hence
\begin{equation}
\label{cota-rho-mas-ve}
  K(\rho+8\delta^{-2}\ve_0)\leq \frac{5\delta^4}{32 C_1 C_2^5}.
\end{equation}
Also, using \eqref{cotas-rho-ve}, we have
\begin{equation}
\label{cota-v-1}
\begin{aligned}
   1+K^2(\rho+\delta^{-1})^2=&
    1+ \frac{C_2^2}{\delta^2}(\rho\delta +1)^2=
   \frac{C^2_2}{\delta^2}\left(  \frac{\delta^2}{C^2_2} +(\rho\delta +1)^2\right) \\
   \leq& \frac{C_2^2}{\delta^2} \left( \frac{\delta^2}{C_2^2} +\left( \frac{\delta^5}{8C_1C_2^6}+1 \right)^2\right)
   \leq
   2\frac{C_2^2}{\delta^2},
 \end{aligned}
 \end{equation}
since $C_1\geq 1$, $C_2\geq 4$ and $\delta\leq 2$.

From the bounds in \eqref{cota-rho-mas-ve} and \eqref{cota-v-1}, combined with
\eqref{cota-u}, \eqref{n-small} and \eqref{semi-v}, we obtain
$$
 [u]_X\leq K(\rho+8\delta^{-2}\ve_0)\leq \frac{5\delta^4}{32 C_1 C_2^5}\leq \frac{5}{2^{11} C_1}
$$
and
$$
 [v]_{X}\leq (1+K^2(\rho+\delta^{-1})^2)^2 [\n]_{X}\leq
 (1+K^2(\rho+\delta^{-1})^2)^2 4K(\rho+8\delta^{-2}\ve_0)\leq
 \left( 2\frac{C_2^2}{\delta^2} \right)^2 \frac{20 \delta^4}{32C_1C_2^5}\leq \frac{5}{8C_1},
$$
since $\delta\leq 2$ and $C_2\geq 4$. Finally, since  $u$ and $v$ are solutions to \eqref{duhamel} with initial condition $u^0$,
\eqref{contraction2} and the above inequalities for $[u]_X$ and $[v]_X$ yield
 \begin{align*}
 \norm{ u-v}_{X}\leq& C_1(2[u]^2_X+[u]_X+[v]_X) \norm{u-v}_{X}
   \\
 \leq& C_1 \left( 2\left( \frac{5}{2^{11} C_1}\right)^2+\frac{5}{2^{11} C_1}+\frac{5}{8C_1} \right)  \norm{u-v}_{X},
 \end{align*}
which implies that  $u=v$, bearing in mind that the constant on the r.h.s. of the above inequality is strictly less that one. This completes the proof of the uniqueness.

It remains to establish \eqref{cont-CI-1}. Let $\m$ and $\n$ two smooth solutions of
\eqref{LLG} satisfying \eqref{cond-small-m}. As a consequence of the uniqueness,
we see that $\m$ and $\n$ are the inverse stereographic projection of some functions $u$
and $v$ that are solutions of \eqref{duhamel} with initial condition $u^0=\P(\m^0)$
and $v^0=\P(\n^0)$, respectively. In particular, $u$ and $v$ satisfy the estimates in \eqref{cota-u}.
Using also \eqref{des-m-n} and \eqref{des-grad-m-n}, we deduce that
\bqq
\begin{aligned}
\norm{\m-\n}_{X} &\leq
3\sup_{t>0}\norm{u-v}_{L^\infty}+4[u-v]_X+12\sup_{t>0}\norm{u-v}_{L^\infty}([u]_X+[v]_X])\\
&\leq 4 \norm{u-v}_X+24C_2(\rho+8\delta^{-2}\ve_0)\norm{u-v}_X,  \\
&\leq 5 \norm{u-v}_X,
\end{aligned}
\eqq
where we have used \eqref{cota-rho-mas-ve} in obtaining the last inequality.
Finally, using also \eqref{cond-CI-1} and \eqref{des-u-v}, and
applying \eqref{cont-CI} in Theorem~\ref{thm:cauchy},
\bqq
\begin{aligned}
\norm{\m-\n}_{X} &\leq 30 K \norm{u^0-v^0}_{L^\infty}\\
&\leq 120 K\delta^{-2}  \norm{\m^0-\n^0}_{L^\infty},
\end{aligned}
\eqq
which yields \eqref{cont-CI-1}.
\end{proof}
\begin{proof}[Proof of Corollary~\ref{cor:cauchy-LLG}]
Let $\boR\in SO(3)$ such that $\boR \QQ=(0,0,-1)$, i.e.\ $\boR$ is the rotation that maps $\QQ$ to the South Pole.
Let us set $\m^0_{\boR}=\boR \m^0$. Then
$$
 \abs{\m^0-\QQ}^2=\abs{\boR(\m^0-\QQ)}^2=
  \abs{\m^0_\boR-(0,0,-1)}^2=2(1+m_{3,\boR}^0).
$$
Hence,
\bqq
\inf_{x\in \R^N} m_{3,\boR}^0\geq -1+\delta
\eqq
and
\bqq
[\m^0_\boR]_{BMO}=[\m^0]_{BMO}\leq \ve_0.
\eqq
Therefore, Theorem~\ref{thm:cauchy-LLG} provides the existence of a unique smooth solution $\m_\boR\in X(\R^N\times \R^+;\S^2)$
of \eqref{LLG} satisfying \eqref{cond-small-m}. Using the invariance of \eqref{LLG}
and setting $\m=\boR^{-1} \m_\boR$ we obtain the existence of the desired solution.
To establish the uniqueness, it suffices to observe that if $\n$ is another smooth solution  of \eqref{LLG}
satisfying \eqref{cond-cor}, then $\n_\boR:=\boR \n$ is a solution of \eqref{LLG} with initial condition
$\m^0_\boR $ and it fulfills  \eqref{cond-small-m}. Therefore, from the uniqueness of solution in Theorem~\ref{thm:cauchy-LLG}, it follows that $\m_\boR=\n_\boR$  and then $\m=\n$.
\end{proof}
\begin{proof}[Proof of Theorem~\ref{thm-cauchy-intro}]
In Theorem~\ref{thm:cauchy-LLG} and Corollary~\ref{cor:cauchy-LLG},
the constants are given by $C=C_1C_2^2$ and $K=C_2$.
As discussed in Remark~\ref{rem-set-delta}, the value
$$\rho^*=\frac{\delta^4}{32C_1C_2^2}$$
maximizes the range for $\ve_0$ in \eqref{condition-rho} and this inequality
is satisfied for any $\ve_0>0$ such that
$$\ve_0\leq \frac{\delta^6}{256C_1C_2^2}.$$
Taking
$$M_1=\frac1{256C_1C_2^2}, \quad M_2=C_2\quad\text{ and }\quad M_3=\frac{1}{32C_1C_2^2},$$
so that $\rho^*=M_3\delta^4$,
the conclusion follows from Theorem~\ref{thm:cauchy-LLG} and Corollary~\ref{cor:cauchy-LLG}.
\end{proof}

\begin{remark} We finally remark that is possible to state  local (in time) versions of Theorems~\ref{thm:cauchy}
and~\ref{thm:cauchy-LLG} as it was done in \cite{koch-tataru,koch-lamm,wang}. In our context, the
local well-posedness would concern solutions  with initial condition $\m^0\in \overline{VMO}(\R^N)$,
i.e.\ such that
\bq\label{vmo}
\lim_{r\to 0^+} \sup_{x\in\R^N}\fint_{B_r(x)}\abs{\m^0(y)-\m^0_{x,r}}\,dy=0.
\eq
Moreover, some uniqueness results  have been established for solutions with this kind of initial data
 by Miura~\cite{miura} for the Navier--Stokes equation, and adapted by Lin~\cite{lin} to  \eqref{HFHM}.
It is also possible to do this for \eqref{LLG}, for $\alpha>0$. We do not pursuit here these types of results  because they
do not apply to the self-similar solutions $\m_{c,\alpha}$. This is due to the facts that the function
$\m^0_{\A^\pm}$ does not belong to $\overline{VMO}(\R)$ and that
\bqq
\lim_{T\to 0^+}\sup_{0<t<T}\sqrt{t}\norm{\partial_x\m_{c,\alpha}}_{L^\infty}\neq 0.
\eqq
\end{remark}
%

%
\section{Applications}
\label{sec-self-similar}

%
\subsection{Existence of self-similar solutions in $\mathbb{R}^N$}\label{sub-self}
%
The LLG equation is invariant under the scaling $(x,t)\to (\lambda x,\lambda^2 t)$, for $\lambda>0$, that is
if $\m$ satisfies \eqref{LLG}, then  so  does the function
\bqq
\m_{\lambda}(x,t)=\m(\lambda x,\lambda^2 t), \qquad \lambda>0.
\eqq
Therefore is natural to study the existence of self-similar solutions (of expander type), i.e.\
a solution $\m$ satisfying
\bq\label{def-self-sim}
\m(x,t)=\m(\lambda x,\lambda^2 t),  \quad \forall \lambda>0,
\eq
or, equivalently,
$$\m(x,t)=\f\left(\frac{x}{\sqrt t}\right),$$
for some $\f:\mathbb{R}^N\longrightarrow \mathbb{S}^2$ profile of $\m$. In particular we have the relation
$\f(y)=\m(y,1)$, for  all $y\in \R^N$.
From \eqref{def-self-sim} we see that, at least formally,
a necessary condition for the existence of a self-similar solution is that
initial condition $\m^0$ be homogeneous of degree 0, i.e.
$$\m^{0}(\lambda x)=\m^0(x), \quad \forall \lambda>0.$$
Since the norm in $X(\R^N\times \R^+;\R^3)$ is invariant under this scaling, i.e.
$$\norm{\m_\lambda}_X=\norm{\m}_X, \quad \forall\lambda>0,$$
where $\m_\lambda$ is defined by \eqref{def-self-sim},
Theorem~\ref{thm:cauchy-LLG} yields the following result concerning the existence
of self-similar solutions.
\begin{corollary}\label{cor-self-similar}
With the same notations and  hypotheses as in Theorem~\ref{thm:cauchy-LLG}, assume also that $\m^0$
is homogeneous of degree zero. Then the solution
$\m$ of \eqref{LLG} provided  by Theorem~\ref{thm:cauchy-LLG} is self-similar.
In particular there exists a smooth profile $\f : \R^N\to \S^2$ such that
$$\m(x,t)=\f\left(\frac{x}{\sqrt t}\right),$$
for all $x\in\R^N$  and $t>0$, and $\f$ satisfies the equation
$$-\frac12y\cdot \grad \f(y)=\beta  \f(y)\times \Delta \f(y) -\alpha \f(y) \times (\f(y)\times \Delta \f(y)),
$$
for all $y\in \R^N$. Here $y\cdot \grad \f(y)=(y\cdot \grad f_1(y),\dots, y\cdot \grad f_N(y))$.
\end{corollary}

\begin{remark}
Analogously, Theorem~\ref{thm:cauchy}
leads to the existence of self-similar solutions for \eqref{DNLS}, provided that $u^0$ is a
homogeneous function of degree zero.
\end{remark}

For instance, in dimensions $N\geq 2$, Corollary~\ref{cor-self-similar} applies to the initial condition
$$\m^0(x)=\H\left(\frac{x}{\abs{x}}\right),$$
with $\H$ a Lipschitz map from $\S^{N-1}$ to $\S^2\cap \{(x_1,x_2,x_3) : x_3\geq -1/2\}$,
provided that the Lipschitz constant is small enough. Indeed, using \eqref{BMO-brezis},
we have
$$[\m^0]_{BMO}\leq 4\norm{\H}_{\text{Lip}},$$
so that taking
$$\delta=1/2,\quad  \rho=\frac{\delta^4}{32K^4C},\quad
\ve_0=\frac{\delta^6}{256K^4C}\quad \text{and }\quad \norm{\H}_{\text{Lip}}\leq \ve_0,$$
the condition  \eqref{cond-LLG} is satisfied
and we can invoke Corollary~\ref{cor-self-similar}.

Other authors have considered self-similar solutions for the harmonic map flow (i.e.\ \eqref{LLG} with  $\alpha=1$) in different settings.
Actually, equation \eqref{HFHM} can be generalized for maps $\m : \mathcal{M}\times \R^+\to \mathcal{N}$,
with $\mathcal M$ and $\mathcal N$ Riemannian manifolds. Biernat and Bizo\'n \cite{biernat-bizon}
established results when  $\mathcal M=\mathcal N=\S^d$ and $3\leq d\leq 6.$
Also, Germain and Rupflin \cite{germain-rupflin} have investigated the case $\mathcal M=\R^d$ and  $\mathcal N=\S^d$,
in $d\geq 3$. In both works the analysis is done only for equivariant  solutions and does not  cover the  case
$\mathcal M=\R^N$ and $\mathcal{N}=\S^2$.

%
\subsection{The Cauchy problem for the one-dimensional LLG equation with a jump initial data}\label{Cauchy}
%

This section is devoted to prove Theorems~\ref{thm:stability} and~\ref{thm-non-uniq} in the introduction. These two results concern the question of well-posedness/ill-posedness of the Cauchy problem for the \emph{one-dimensional}
LLG equation associated with a step function initial condition of the form
\bq\label{def-m0}
\m^0_{\A^\pm}:=\A^+\chi_{\R^+}+\A^{-}\chi_{\R^-},
\eq
where  $\A^+$ and $\A^-$ are two given unitary vectors in $\S^2$.

%
\subsubsection{Existence, uniqueness and stability. Proof of Theorem~\ref{thm:stability}}\label{sub-self-1}
As mentioned in the introduction, in \cite{gutierrez-delaire} we proved the existence of the uniparametric smooth family of self-similar solutions 
$\{\m_{c,\alpha}\}_{c>0}$ of \eqref{LLG} for all  \mbox{$\alpha\in[0,1]$} with initial condition of the type \eqref{def-m0} given by
\bq\label{def-m0-c}
 \m^0_{c,\alpha}:=\A^+_{c,\alpha}\chi_{\R^+}+\A^-_{c,\alpha}\chi_{\R^-},
\eq
where $\A^{\pm}_{c,\alpha}\in\S^2$ are given by Theorem~\ref{thm-self}.
For the convenience of the reader, we collect some of the results
proved in  \cite{gutierrez-delaire} in the Appendix. The results in this section rely on a further understanding of the properties of the self-similar solutions $\m_{c,\alpha}$.

In Proposition~\ref{cotas-self} we show that
$$\m_{c,\alpha}=(m_{1,c,\alpha},m_{2,c,\alpha},m_{3,c,\alpha}) \in X(\R\times \R^+;\S^2),$$
that  $m_{3,c,\alpha}$ is far from the South Pole and that $[\m_{c,\alpha}]_X$ is small, if $c$ is small enough.
This will yield that  $\m_{c,\alpha}$ corresponds (up to a rotation) to the solution given by Corollary~\ref{cor-self-similar}.
More precisely, using the invariance under rotations of \eqref{LLG}, we can prove that, if the angle between $\A^+$ and $\A^-$ is small enough, then
the solution given by Corollary~\ref{cor-self-similar} with initial condition $\m^0_{\A^\pm}$ coincides (modulo a rotation) with
$\m_{c,\alpha}$, for some $c$. We have the following:

\begin{thm}\label{thm-small-angle}
Let $\alpha\in (0,1]$. There exist $L_1, L_2>0$, $\delta^*\in(-1,0)$ and $\vartheta^*>0$  such that the following holds. Let  $\A^+$, $\A^- \in \S^2$ and let $\vartheta$ be the angle between them.
If
\bq\label{cond-angle}
0<\vartheta\leq \vartheta^*,
\eq
then there exists a solution $\m$ of \eqref{LLG}
with initial condition $\m^0_{\A\pm}$. Moreover, there exists $0<c<\frac{\sqrt{\alpha}}{2\sqrt{\pi}}$, such that
$\m$ coincides up to a rotation with the self-similar solution  $\m_{c,\alpha}$, i.e.\
there exists $\boR\in SO(3)$, depending only on $\A^+$, $\A^-$, $\alpha$ and $c$, such that
\bq
\m=\boR \m_{c,\alpha},
\eq
and $\m$ is the unique solution satisfying
\begin{equation}\label{unicidad-self}
 \inf_{\substack{x\in \R\\ t>0}} m_3(x,t)\geq \delta^*
 \quad \textup{ and }\quad [\m]_{X}\leq L_1+L_2 c.
\end{equation}
\end{thm}

In order to prove Theorem~\ref{thm-small-angle}, we need some preliminary estimates for $\m_{c,\alpha}$
in terms of $c$ and $\alpha$. To obtain them, we use some properties of the profile profile
$\f_{c,\alpha}=( f_{1,c,\alpha}, f_{2,c,\alpha}, f_{3,c,\alpha})$
 constructed in  \cite{gutierrez-delaire} using the Serret--Frenet equations with initial conditions
  \bqq
  f_{1,c,\alpha}(0)=1,\quad f_{2,c,\alpha}(0)=f_{3,c,\alpha}(0)=0.
  \eqq
Also,
 \bqq
  \abs{f'_{j,c,\alpha}(s)}  \leq ce^{-\alpha s^2/4}, \text{ for all }s\in \R,
  \eqq
  for $j\in\{1,2,3\}$ and
 \bq\label{m-f}
\m_{c, \alpha}(x,t) =\f_{c,\alpha}\left( \frac{x}{\sqrt{t}}  \right),  \quad \text{for all }(x,t)\in\R\times \R^+.
\eq
Hence, for any $x\in \R$,
$$
\abs{f_{3,c_\alpha}(x)}=\abs{f_{3,c_\alpha}(x)-f_{3,c_\alpha}(0)}\leq \int_{0}^{\abs{x}}ce^{-\alpha \sigma^2/4}d\sigma\leq
c\frac{\sqrt{\pi}}{\sqrt\alpha}.
$$
Since the same estimate holds for $f_{2,c,\alpha}$, we conclude that
 \bq\label{cota-m3}
 \abs{m_{2,c,\alpha}(x,t)}\leq c\frac{\sqrt{\pi}}{\sqrt\alpha},\quad \text{ and }\quad \abs{m_{3,c,\alpha}(x,t)}\leq c\frac{\sqrt{\pi}}{\sqrt\alpha} \quad \textup{ for all } (x,t)\in \R\times \R^+.
 \eq
Moreover, since
$$\A_{c,\alpha}^{\pm}=\lim_{x\to\pm\infty}\f_{c,\alpha}(x),$$
we also get
\bq\label{cota-A}
\abs{A^\pm_{j,c,\alpha}}\leq  c\frac{\sqrt{\pi}}{\sqrt\alpha}, \quad \textup{ for } j\in\{2,3\}.
\eq
We now provide some further properties of the self-similar solutions.

\begin{prop}\label{cotas-self}
For  $\alpha\in (0,1]$ and $c>0$, we have
\begin{gather}
\label{cond-4}
\norm{m_{2,c,\alpha}^0}_{L^\infty}\leq c\frac{\sqrt{\pi}}{\sqrt\alpha}, \quad
\norm{m_{3,c,\alpha}^0}_{L^\infty}\leq c\frac{\sqrt{\pi}}{\sqrt\alpha}, \quad
  \sup_{t>0} \norm{m_{3,c,\alpha}}_{L^\infty}\leq  c\frac{\sqrt{\pi}}{\sqrt\alpha},\\
  \label{cond-1} [\m^0_{c,\alpha}]_{BMO}\leq  2c\frac{\sqrt{2\pi}}{\sqrt\alpha},\\
\label{cond-2} \sqrt t\norm{\partial_{x} \m_{c,\alpha}}_{\infty}=c, \quad \text{for all }t>0,\\
\sup_{\substack{x\in\R\\r>0}}\label{cond-3}\frac1{r}\int_{Q_r(x)} \abs{\partial_y \m_{c,\alpha} (y,t) }^2\,dt\,dy \leq  \frac{2\sqrt{2\pi}c^2}{\sqrt{\alpha}}.
\end{gather}
In particular, $\m_{c,\alpha}\in X(\R\times \R^+;\S^2)$ and
\bq\label{semi-norm-self}
[\m_{c,\alpha}]_{X}\leq \frac{4 c}{\alpha^{\frac{1}{4}}}.
\eq
\end{prop}		
\begin{proof}[Proof of Proposition~\ref{cotas-self}]
The estimates in \eqref{cond-4} follow from \eqref{cota-m3} and \eqref{cota-A}.
To prove \eqref{cond-1}, we use \eqref{BMO-brezis},
 \eqref{def-m0-c}, \eqref{cond-4}  and the fact that
 \bq\label{Amenos}
 \A^-_{c,\alpha}=(A_{1,c,\alpha}^+,-A_{2,c,\alpha}^+,-A_{3,c,\alpha}^+),
 \eq
(see Theorem~\ref{thm-self}) to get
\begin{align*}
 [\m^0_{c,\alpha}]_{BMO}
 &\leq \sup_{x\in\R^N}\fint_{B_r(x)}\fint_{B_r(x)}\abs{\m^0_{c,\alpha}(y)-\m^0_{c,\alpha}(z)}\,dy\,dz\\ \nonumber
 &\leq 2\sqrt{ (A^+_{2,c,\alpha})^2+(A^+_{3,c,\alpha})^2}\sup_{x\in\R^N}\fint_{B_r(x)}\fint_{B_r(x)}\,dy\,dz\\  \nonumber
 &\leq \frac{2c\sqrt{2\pi}}{\sqrt{\alpha}}. \nonumber
\end{align*}

From \eqref{derivada} we obtain the equality in  \eqref{cond-2} and  also
 \bq\label{1}
 I_{r,x}:=\frac1{r}\int_{Q_r(x)} \abs{\partial_y \m_{c,\alpha} (y,t) }^2\,dt\,dy
 =\frac{c^2}{r}\int_{x-r}^{x+r}\int_0^{r^2}\frac{ e^{\frac{-\alpha y^2}{2t}} } {t} \,dt\,dy.
 \eq
Performing the change of variables $z=(\alpha y^2)/(2t)$, we see that
\bq\label{2}
\int_0^{r^2}\frac{ e^{\frac{-\alpha y^2}{2t}} } {t} \,dt=E_1\left(\frac{\alpha y^2}{2r^2}\right),
\eq
where $E_1$ is the exponential integral function
$$E_1(y)=\int_{y}^\infty\frac{e^{-z}}{z}\,dz.$$
This function satisfies that $\lim_{y \to 0^+}E_1(y)=\infty$ and $\lim_{y \to \infty}E_1(y)=0$ (see e.g. \cite[Chapter 5]{abram}).
Moreover, taking $\epsilon>0$ and integrating by parts,
\bq\label{proof-ipp}
\int_{\epsilon}^\infty E_1(y^2)\,dy=\left. yE_1(y^2)\right|^{\infty}_{\epsilon}+2\int_{\epsilon}^\infty e^{-y^2}\,dy,
\eq
so L'H\^opital's rule shows that the first term in the r.h.s.\ of \eqref{proof-ipp} vanishes as $\epsilon\to 0^+$.
Therefore, the Lebesgue's monotone convergence theorem allows to conclude that $E_1(y^2)\in L^1(\R^+)$
and
\bq\label{E1}
\int_{0}^\infty E_1(y^2)=\sqrt{\pi}.
\eq
By using  \eqref{1}, \eqref{2}, \eqref{E1},
and making the change of variables $z=\sqrt{\alpha}y/(r\sqrt{2})$, we obtain
\bq\label{Ixr}
I_{r,x}=\frac{c^2}{r}\int_{x-r}^{x+r}E_1\left(\frac{\alpha y^2}{2r^2}\right)\,dy=
\frac{\sqrt{2}c^2}{\sqrt{\alpha}}
\int_{ \frac{\sqrt{\alpha}}{\sqrt{2}}(\frac{x}{r}-1)}^{ \frac{\sqrt{\alpha}}{\sqrt{2}}(\frac{x}{r}+1)}E_1(z^2)\,dz\leq
\frac{\sqrt{2}c^2}{\sqrt{\alpha}}\cdot 2\sqrt{\pi},
\eq
which leads to \eqref{cond-3}.
Finally, the bound in \eqref{semi-norm-self} easily follows from those in \eqref{cond-2} and \eqref{cond-3} and the elementary inequality
\bqq\Big( 1+
\Big(\frac{2\sqrt{2\pi}}{\sqrt{\alpha}}\Big)^{1/2}
\Big)\leq
\frac{1}{\alpha^{\frac{1}{4}}} \big( 1+
({2\sqrt{2\pi}})^{1/2}
\big)
\leq \frac{4}{\alpha^{\frac{1}{4}}}, \qquad \alpha\in(0,1].
\eqq
\end{proof}
%
\begin{proof}[Proof of Theorem~\ref{thm-small-angle}]
First, we consider the case when $\A^+=\A^+_{c,\alpha}$
and $\A^-=\A^-_{c,\alpha}$ (i.e.\ when  $\m^0_{\A^{\pm}}=\m^0_{c,\alpha}$) for some $c>0$. We will continue to show that the solution provided
by Theorem~\ref{thm:cauchy-LLG}  is exactly $\m_{c,\alpha}$, for $c$  small. Indeed, bearing in mind the estimates in Proposition~\ref{cotas-self},
we consider
$$c\leq \frac{\sqrt{\alpha}}{2\sqrt{\pi}},$$
so that
\bq\label{cota-inf-m}
\inf_{x\in \R} m^0_{3,c,\alpha}(x)\geq -\frac12.
\eq
In view of \eqref{cond-1}, \eqref{cota-inf-m}  and Remark~\ref{rem-set-delta}, we set
\bq\label{parametros}
\ve_0:=4c\frac{\sqrt \pi}{\sqrt{\alpha}}, \quad  \delta:=\frac12,\quad \rho:=\frac{\delta^4}{8K^4C}=\frac1{2^7K^4C},
\eq
where $C,K\geq 1$ are the constants given by Theorem~\ref{thm:cauchy-LLG}.
In this manner, from \eqref{cond-1}, \eqref{cota-inf-m} and \eqref{parametros}, we have
\bqq
\inf_{\R}{m^0_3}\geq -1+\delta \quad\textup{ and }\quad[\m^0]_{BMO}\leq \ve_0,
\eqq
and the condition \eqref{cond-LLG} is fulfilled
if $$\ve_0\leq \frac{\delta^6}{256 K^4 C},$$
or equivalently, if
$c\leq \tilde c$,  with
$$\tilde c:=\frac{\sqrt\alpha}{2^{16}K^4C\sqrt\pi}.
$$
Observe that in particular $\tilde c<\frac{\sqrt{\alpha}}{2\sqrt{\pi}}$.

For fixed $0<c<\tilde c$, we can apply Theorem~\ref{thm:cauchy-LLG}
to deduce the existence and uniqueness of a solution $\m$ of \eqref{LLG} satisfying
 \bq\label{cond-small-m-1}
 \inf_{\substack{x\in \R\\ t>0}} m_3(x,t)\geq -1+\frac{2}{1+K^2(\rho+2)^2}
 \quad \textup{ and }\quad [\m]_{X}\leq 4K\rho+\frac{2^9Kc\sqrt{\pi}}{\sqrt{\alpha}}.
 \eq
 Now by Proposition~\ref{cotas-self}, for fixed $0<c\leq \tilde c$, we have the following estimates for $\m_{c, \alpha}$
$$
 [\m_{c,\alpha}]_{X}\leq 4c/\alpha^{\frac{1}{4}} \qquad \text{and}\qquad
 \inf_{\substack{x\in \R\\ t>0}} m_{3,c,\alpha}(x,t)\geq -\frac{1}{2},
$$
so in particular $\m_{c,\alpha}$ satisfies \eqref{cond-small-m-1}. Thus the uniqueness of solution implies that
$\m=\m_{c,\alpha}$, provided that $c\leq \tilde c$.
Defining the constants  $L_1$,  $L_2$ and $\delta^*$ by
\bq\label{def-L1}
  L_1=4K\rho, \qquad L_2=\frac{2^9 K \sqrt{\pi}}{\sqrt{\alpha}} \qquad {\text{and}}\qquad  \delta^*=-1+\frac{2}{1+K^2(\rho+2)^2},
\eq
the theorem is proved in the case $\A^\pm=\A^\pm_{c,\alpha}$.

For the general case, we would like to understand which angles can be reached by varying the parameter $c$ in the range  $(0,\tilde c]$. 
To this end, for fixed $0<c\leq \tilde c$, let $\vartheta_{c,\alpha}$ be the angle between $\A^{+}_{c,\alpha}$ and $\A^{-}_{c,\alpha}$. From Lemma~\ref{lemma-theta},
$$
 \vartheta_{c,\alpha}\geq \arccos\left( 1-c^2\pi+32 \frac{c^3\sqrt{\pi}}{\alpha^2}  \right), \qquad \text{ for all }
  c\in\Big(0, \frac{\alpha^2\sqrt{\pi}}{32}\Big].
$$
Now, it is easy to see that the function $F(c)=\arccos\left( 1-c^2\pi+32 \frac{c^3\sqrt{\pi}}{\alpha^2}  \right)$ is strictly increasing on the interval $[0,\alpha^2\frac{\sqrt{\pi}}{48}]$ so that
\begin{equation}\label{F}
  F(c)>F(0)=0, \qquad \text{ for all } c\in\Big(0,\frac{\alpha^2\sqrt{\pi}}{48}\Big].
\end{equation}
Let $c^*=\min(\tilde c, \frac{\alpha^2\sqrt{\pi}}{48})$ and consider the map $T_\alpha: c\longrightarrow \vartheta_{c,\alpha}$ on $[0, c^*]$.
By  Lemma~\ref{lemma-theta}, $T_\alpha$ is continuous on $[0,c^*]$, $T_\alpha(0)=\lim_{c\rightarrow 0^{+}} T_\alpha(c)=0$ and, bearing in mind \eqref{F}, $T(c^*)=\vartheta_{c^*,\alpha}>0$. 
Thus, from the intermediate value theorem we infer that for any $\vartheta\in(0,\vartheta_{c^*,\alpha})$, there exists $c\in (0,c^*)$ such that
$$
\vartheta=T_\alpha(c)=\vartheta_{c,\alpha}.
$$
We can now complete the proof for any  $\A^+$, $\A^{-}\in\mathbb{S}^2$.  Let $\vartheta$ be the angle between $\A^+$ and $\A^-$. From the previous lines, we know that there exists $\vartheta^*:=\vartheta_{c^*,\alpha}$ such that if $\vartheta\in (0,\vartheta^*)$, there exists $c\in(0,c^*)$ such that $\vartheta=\vartheta_{c,\alpha}$. For this value of $c$, consider the initial value problem associated with $\m^0_{c,\alpha}$
and the constants defined in \eqref{def-L1}. We have already seen the existence of a unique solution $\m_{c,\alpha}$ of the LLG equation associated with this initial condition satisfying \eqref{unicidad-self}. Let $\boR\in SO(3)$ be the rotation on $\mathbb{R}^3$ such that $\A^+=\boR\A^{+}_{c,\alpha}$ and $\A^-=\boR\A^{-}_{c,\alpha}$. Then $\m:=\boR\m_{c,\alpha}$ solves \eqref{LLG} with initial condition $\m^0_{\A^\pm}$. Finally, recalling the above definition of  $L_1$, $L_2$ and $\delta^*$,  using the invariance of the norms under rotations and the fact that $\m_{c,\alpha}$ is the unique solution satisfying \eqref{cond-small-m-1}, it follows that $\m$ is the unique solution satisfying the conditions in the statement of the theorem.
\end{proof}

We are now  in position to give the proof of Theorem~\ref{thm:stability}, the second of our main results in this paper. In fact, we will
 see that Theorem~\ref{thm:stability} easily follows from  Theorem~\ref{thm-small-angle} and the well-posedness for the LLG equation stated in Theorem~\ref{thm:cauchy-LLG}.
\medskip

\begin{proof}[Proof of Theorem~\ref{thm:stability}]
Let $\vartheta^*$, $\delta^*$, $L_1$ and $L_2$ be the constants defined in the proof of  Theorem~\ref{thm-small-angle}.
Given $\A^+$ and $\A^-$ such that $0<\vartheta<\vartheta^*$, Theorem~\ref{thm-small-angle} asserts the existence of
\bq\label{range-c}
0<c<\frac{\sqrt{\alpha}}{2\sqrt{\pi}}
 \eq
 and $\boR\in SO(3)$ such that $\boR \m_{c,\alpha}$ is the unique solution of (LLG$_\alpha$) with initial condition $\m^{0}_{\A^{\pm}}$ satisfying \eqref{unicidad-self}, and in particular $\m^{0}_{\A^{\pm}}=\boR\m^{0}_{c,\alpha}$.
By hypothesis  $\m^0$ satisfies
\begin{equation}\label{cond-h-1}
\norm{\m^0-\m^0_{\A^\pm} }_{L^\infty}\leq \frac{c\sqrt{\pi}}{2\sqrt{\alpha}}  .
\end{equation}
Hence, defining $\m^0_\boR=\boR^{-1} \m^0$, recalling that $[f]_{BMO}\leq 2\norm{f}_{L^\infty}$
and using the invariance of the norms under rotations, we deduce from \eqref{cond-h-1} that
$$
\norm{ \m^0_\boR }_{L^\infty}\leq \norm{ \m^0_{c,\alpha}}_{L^\infty} + \frac{c\sqrt{\pi}}{2\sqrt{\alpha}}
\quad\text{ and }\quad
[\m^0_\boR ]_{BMO}\leq [\m^0_{c,\alpha}]_{BMO} + \frac{c\sqrt{\pi}}{\sqrt{\alpha}}.
$$
Then, by Proposition~\ref{cotas-self},
\bq\label{MoR}
\norm{ \m^0_{3,\boR} }_{L^\infty}\leq\frac{2c\sqrt{\pi}}{\sqrt{\alpha}}
\quad\text{ and }\quad
[\m^0_\boR ]_{BMO}\leq \frac{4c\sqrt{\pi}}{\sqrt{\alpha}}.
\eq
From \eqref{range-c} and \eqref{MoR}, it follows that
$$
\m^0_{3,\boR}(x)\geq -1/2,\qquad \text{ for all } x\in \R.
$$
Therefore, as in the proof of Theorem~\ref{thm-small-angle},
we can apply Theorem~\ref{thm:cauchy-LLG} with the values of $\ve_0$, $\delta$ and $\rho$ given in \eqref{parametros}
to deduce the existence of a unique (smooth) solution $\m_\boR$ of \eqref{LLG} with initial condition  $\m^0_\boR$
satisfying
$$
  \inf_{\substack{x\in \R\\ t>0}} \m_{3,\boR}(x,t)\geq -1+\frac{2}{1+K^2(\rho+2)^2}=\delta^*
 \quad \textup{ and }\quad [\m_\boR]_{X}\leq 4K\rho+\frac{2^{9}Kc\sqrt{\pi}}{\sqrt{\alpha}}=L_1+L_2 c.
$$
Since we have taken the values for  $\ve_0$, $\delta$ and $\rho$ as in the proof
Theorem~\ref{thm-small-angle}, Theorem~\ref{thm:cauchy-LLG} also implies that
\bqq
\norm{ \m_\boR-\m_{c,\alpha}}_{X}\leq  480K \norm{\m^0_\boR-\m^0_{c,\alpha} }_{L^\infty}.
\eqq
The conclusion of the theorem follows defining $\m=\boR\m_{\boR}$ and  $L_3=480K$, and using once again the invariance of the norm under rotations.
\end{proof}
%
\subsubsection{Multiplicity of solutions. Proof of Theorem~\ref{thm-non-uniq}}\label{sub-self-2}

As proved in \cite{gutierrez-delaire}, when $\alpha=1$, the self-similar solutions are explicitly given by
\bq\label{explicit-for}
\m_{c,1}(x,t)=( \cos(c\Erf(x/\sqrt t)), \sin(c\Erf(x/\sqrt t)),0), \quad \text{ for all }(x,t)\in\R\times \R^+,
\eq
for every $c>0$, where $\Erf(\cdot)$ is the non-normalized error function
$$
 \Erf(s)=\int_0^s e^{-\sigma^2/4}\, d\sigma.
$$
In particular,
$$\vec \A^\pm_{c,1}=(\cos(c\sqrt\pi),\pm \sin(c\sqrt\pi),0 )$$
and the angle between $\A^+_{c,1}$ and $\A^-_{c,1}$ is given by
\bq\label{explicit-for2}
\vartheta_{c,1}=\arccos(\cos(2c\sqrt{\pi})).
\eq
\begin{figure}[ht!]
\begin{center}
  \begin{overpic}[trim=0 0 30 0,clip,scale=0.7]{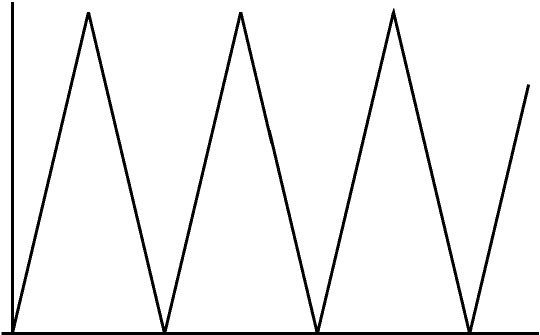}
\put(4,71){\small{$\vartheta_{c,1}$}}
\put(-2,64){\small{$\pi$}}
\put(45,-5){\small{$c$}}
\end{overpic}
  \end{center}
  \caption{The angle $\vartheta_{c,\alpha}$ as a function of $c$ for $\alpha=1$.}
\label{fig-angulo}
\end{figure}

Formula \eqref{explicit-for2} and Figure~\ref{fig-angulo} show that there are infinite values of $c$
that allow to reach any angle in $[0,\pi]$.
Therefore, using the  invariance of \eqref{LLG} under rotations, in the case when $\alpha=1$, one can easily prove the existence of multiple solutions associated with a given initial data of the form $\m^{0}_{\A^{\pm}}$ for any given vectors $\A^{\pm}\in \mathbb{S}^2$ (see argument in the proof included below). In the case that $\alpha$ is close enough to $1$, we can use a continuity argument
to prove that we still have multiple solutions. More precisely, Theorem~\ref{thm-non-uniq} asserts that for any given initial data of the form $\m^{0}_{\A^{\pm}}$ with angle between $\A^{+}$ and $\A^{-}$ in the interval $(0,\pi)$, if $\alpha$ is sufficiently close to one, then there exist {\it{at least}} $k$-distinct solutions of \eqref{LLG} associated with the same initial condition, for any given $k\in\mathbb{N}$.

The rest of this section is devoted to the proof of Theorem~\ref{thm-non-uniq}.

\begin{proof}[Proof of Theorem~\ref{thm-non-uniq}]
Let $k\in \N$,  $\A^\pm \in \S^2$ and $\vartheta\in (0,\pi)$ be the angle between $\A^+$ and  $\A^-$. Using the invariance of  \eqref{LLG} under rotations, it suffices to prove the existence of $\alpha_k\in (0,1)$  such that
for every $\alpha\in[\alpha_k,1]$ there exist $0<c_1<\dots<c_k$ such that
the angle $\vartheta_{c_j,\alpha}$ between $\A^+_{c_j,\alpha}$ and  $\A^-_{c_j,\alpha}$, satisfies
\bq\label{theta-l}
\vartheta_{c_j,\alpha}=\vartheta, \quad \text{ for all }j\in\{1,\dots,k\}.
\eq
In what follows, and since we want to show the existence of at least $k$-distinct solutions, we will assume without loss of generality that $k$ is large enough.

First observe that, since $\A^-_{c,\alpha}=(A_{1,c,\alpha}^+,-A_{2,c,\alpha}^+,-A_{3,c,\alpha}^+)$, we have the explicit
formula
$$\cos(\vartheta_{c,\alpha})=2(A^+_{1,c,\alpha})^2-1,$$
and using  Lemma~\ref{lemma-A} in the Appendix, we get
\bq\label{cota-theta}
\abs{\cos(\vartheta_{c,\alpha})-\cos(\vartheta_{c,1})}=\abs{2((A^+_{1,c,\alpha})^2-(A^+_{1,c,1})^2)}\leq 
4\abs{A^+_{1,c,\alpha}-A^+_{1,c,1}}\leq
4h(c)\sqrt{1-\alpha},
\eq
for all $\alpha\in[1/2,1]$, with $h:\mathbb{R}^{+}\longrightarrow \mathbb{R}^{+}$ an increasing function satisfying $\lim_{s\rightarrow \infty} h(s)=\infty$.

For $j\in \N$, we set $a_j=(2j+1)\sqrt{\pi}/2$ and $b_j=(2j+2)\sqrt{\pi}/2$, so
that \eqref{explicit-for2} and \eqref{cota-theta} yield
\bq\label{cota-theta-3}
\cos(\vartheta_{a_j,\alpha})\leq -1+4h(a_j)\sqrt{1-\alpha} \quad\text{ and }\quad
\cos(\vartheta_{b_j,\alpha})\geq 1-4h(b_j)\sqrt{1-\alpha}, \qquad \forall \alpha\in[1/2,1].
\eq
Define $l=\cos(\vartheta)$ and
$$
\alpha_{k}=\max\left( 1-\Big(\frac{1-l}{8h(b_k)}\Big)^2,1-\Big(\frac{1+l}{8h(b_k)}\Big)^2\right).
$$
Notice that, since $\vartheta\in(0,\pi)$, we have  $-1<l<1$ and thus $\alpha_k<1$. Also, since $h$ diverges to $\infty$, we can assume without loss of generality that $\alpha_k\in[1/2,1)$, and from the definition of $\alpha_k$ we have
$$
 0<\sqrt{1-\alpha_k}<\min\left( \frac{1-l}{8h(b_k)}, \frac{1+l}{8h(b_k)}  \right).
$$
Therefore, from \eqref{cota-theta-3} and  $h(a_j)<h(b_j)\leq h(b_k)$ (since $h$ is a strictly increasing function), we get
$$
\cos(\vartheta_{a_j,\alpha})\leq \frac{-1+l}{2} \quad\text{ and }\quad
 \frac{1+l}{2} \leq \cos(\vartheta_{b_j,\alpha}), \quad \forall j\in\{1,\dots,k\}, \ \forall \alpha\in[\alpha_k,1],
$$ and  thus
\begin{equation} \label{cota-theta-2}
\cos(\vartheta_{a_j,\alpha})< l<\cos(\vartheta_{b_j,\alpha}), \quad \forall j\in\{1,\dots,k\}, \ \forall \alpha\in[\alpha_k,1],
\end{equation}
since $l\in(-1,1)$.

Let us fix $\alpha\in[\alpha_k,1]$ and  $j\in\{1, \dots, k\}$. By Lemma~\ref{lemma-A}, $c\rightarrow \cos(\vartheta_{c,\alpha})$ is a continuous function on $c$. Therefore
\eqref{cota-theta-2} and the intermediate value theorem yield the existence  of $c_j\in[a_j,b_j]$
such that
$$
  \cos(\vartheta_{c_j,\alpha})=l=\cos(\vartheta),
$$
or equivalently,  such that $\vartheta_{c_j,\alpha}=\vartheta$.

Finally, for each $j\in\{1,\dots,k\}$, let $\boR_j\in SO(3)$ be
such that $\m^0_{A^\pm}=\boR_j \m^0_{c_j,\alpha}$, and define $\m_j$ by
 $\m_j=\boR_j \m_{c_j,\alpha}$. Then  $\m_j$   solves \eqref{LLG} with initial data $\m^{0}_{\A^\pm}$ and the control of $\partial_x\m_j$ in \eqref{derivada-mk} follows from the definition of $\m_j$ in terms of $\m_{c_j,\alpha}$ and the analogous property established in Theorem~\ref{thm-self}  for the self-similar solution $\m_{c_j,\alpha}$  (see~\eqref{derivada}).

In the case when $\alpha=1$ and $\vartheta\in[0,\pi]$, formula \eqref{explicit-for2} shows that 
sequence $\{c_j\}_{j\geq 1}$ in \eqref{formula-c-j} satisfies  $\vartheta_{c_j,1}=\vartheta$ for all $j\in\mathbb{N^*}$. 
The result follows by considering the sequence of solutions $\{\m_j\}_{j\geq 1}$  described  above.
\end{proof}
\begin{remark}
Notice that the proof given above also shows how close $\alpha_k$  needs to be to 1
in terms of the (fixed) angle $\vartheta\in(0,\pi)$, and in particular $\alpha_k\rightarrow 1$ as $k\rightarrow \infty$, and $\alpha_k\rightarrow 1$ as $\vartheta\rightarrow 0$ (i.e.\ when $l\rightarrow 1$).
\end{remark}

\begin{remark}
For  $\alpha=1$ and $c>0$, the function
$$u(x,t)=\P(\m_{c,1})=\exp\big(ic \Erf(x/\sqrt t)  \big)$$
is a solution of \eqref{DNLS} with initial condition
$$u^0=e^{ic\sqrt{\pi}} \chi_{\R^+}+e^{-ic\sqrt{\pi}} \chi_{\R^-}.$$
Therefore there is also a multiplicity phenomenon for the equation \eqref{DNLS}.
\end{remark}

\subsection{A singular solution for a nonlocal Schr\"odinger equation}
\label{sec-singular}
%
We have used the stereographic projection to establish a well-posedness result for \eqref{LLG}.
Melcher \cite{melcher} showed a global well-posedness result, provided that 
$$\norm{\grad \m^0}_{L^N}\leq \ve, \quad \m^0-\QQ\in H^1(\R^N)\cap W^{1,N}(\R^N), \quad \alpha>0,\quad  N\geq 3,$$
for some $\QQ\in\mathbb{S}^2$ and $\ve>0$ small. Later,  Lin, Lan and Wang~\cite{lin-lai-wang} 
improved this result and proved  global well-posedness under the conditions 
$$\norm{\grad\m^0}_{M^{2,2}}\leq \ve, \quad \m^0-\QQ\in L^2(\R^N), \quad \alpha>0, \quad N\geq 2,$$
for some $\QQ\in\S^2$ and $\ve>0$ small.\footnote{We recall that $v\in M^{2,2}(\R^N)$ if $v\in L^2_\loc(\R^N)$ and 
$$\norm{v}_{M^{2,2}}:= \sup_{\substack{ x\in \R^N\\ r>0}} \frac{1}{r^{(N-2)/2}}\norm{v}_{L^2(B_r(x))}<\infty.$$}
In the context of Theorem~\ref{thm-cauchy-intro} and using the characterization of $BMO^{-1}$
in Theorem~\ref{tataru}, the second condition in \eqref{CI-intro} says that $\norm{\grad \m^0}_{BMO^{-1}}$ is small.
In view of the embeddings 
\bqq
L^N(\R^N)\subset M^{2,2}(\R^N)\subset BMO^{-1}(\R^N),
\eqq
for $N\geq 2$, we deduce that Theorem~\ref{thm-cauchy-intro} includes initial conditions with less regularity, 
as long as their essential range is not $\S^2$. 
The argument in \cite{lin-lai-wang,melcher}  is
based on the method of moving frames that produces a covariant
complex Ginzburg--Landau equation. 
One of the aims of this subsection is to compare their approach in the context of the self-similar solutions $\m_{c,\alpha}$,
and in particular to draw attention to a possible difficulty in using it to study these solutions.

In the sequel we consider the one-dimensional case $N=1$ and $\alpha\in[0,1]$.
Then the moving frames technique can be recast as a Hasimoto transformation as follows.
Assume that $\m$ is the tangent vector of a curve in $\R^3$, i.e.\ $\m=\ptl_{x}\XX$, for some curve $\XX(x,t)\in
\R^3$ parametrized by the arc-length. It can be shown
(see \cite{daniel-lak}) that if $\m$ evolves under \eqref{LLG}, then the torsion $\tau$ and the curvature $\cc$
of $\XX$ satisfy
\begin{align*}
\partial_t\tau=&\beta\left(\cc \partial_x \cc+\partial_x \Big(\frac{ \partial_{xx} \cc-\cc\tau^2}{\cc}\Big) \right)
+\alpha\left(\cc^2\tau+
\partial_x \Big(\frac{\partial_x (\cc\tau)+\tau\partial_x \cc}{\cc}\Big)
\right), \\
\partial_t \cc=&\beta\left(-\partial_x (\cc\tau)-\tau \partial_x \cc \right) +\alpha\left(\partial_x \cc-\cc\tau^2  \right).
\end{align*}
Hence, defining the Hasimoto transformation \cite{hasimoto} (also called filament function)
\bq\label{hasimoto}
v(x,t)=\cc(x,t)\displaystyle{e^{i\int_0^x\tau(\sigma,t)\,d\sigma}},
\eq
we verify that $v$ solves the following dissipative Schr\"odinger (or complex Ginzburg--Landau) equation
\bq\label{schrodinger}
i\partial_t v+(\beta-i\alpha) \partial_{xx} v+\frac{v}{2}\left(\beta\abs{v}^2+2\alpha\int_0^x
\Im(\bar v \partial_x v)-A(t)\right)=0,
\eq
where $\beta=\sqrt{1-\alpha^2}$ and
\bqq
A(t)=\left(\beta\left(\cc^2+\frac{2(\partial_{xx} \cc-\cc\tau^2)}{\cc}\right)+2\alpha\left(\frac{ \partial_{x}(\cc\tau)+\tau\partial_{x}\cc}{\cc}\right)\right)(0,t).
\eqq
The curvature and torsion associated with the self-similar solutions $\m_{c,\alpha}$ are (see \cite{gutierrez-delaire}):
\begin{equation} \label{cur-tor}
 \cc_{c,\alpha}(x,t)=\frac{c}{\sqrt{t}}e^{-\frac{\alpha x^2}{4t}}
 \qquad {\hbox{and}} \qquad
 \tau_{c,\alpha}(x,t)= \frac{\beta x}{2\sqrt{t}}.
\end{equation}
Therefore in this case
\bq\label{A(t)}
A(t)= \frac{\beta c^2}{t}
\eq
and the Hasimoto transformation of $\m_{c,\alpha}$ is
$$
v_{c,\alpha}(x,t)=\frac{c}{\sqrt{t}}e^{(-\alpha+i\beta)\frac{x^2}{4t}}.
$$
In particular $v_{c,\alpha}$ is a solution of \eqref{schrodinger} with $A(t)$ as in \eqref{A(t)},
for all $\alpha\in[0,1]$ and $c>0$. Moreover, the Fourier transform of this function (w.r.t.\ the space variable)
is
\bqq
\widehat v_{c,\alpha}(\xi,t)=2c\sqrt{\pi(\alpha+i\beta)}  e^{-(\alpha+i\beta)\xi^2t},
\eqq
so that $v_{c,\alpha}$ is a solution of \eqref{schrodinger} with a Dirac delta as initial condition:
\bqq
v_{c,\alpha}(\cdot,0)=2c\sqrt{\pi(\alpha+i\beta)}\delta.
\eqq
Here $\delta$ denotes the delta distribution at the point $x=0$
and $\sqrt{z}$ denotes the square root of a complex number $z$ such
that $\Im (\sqrt{z})>0$.

In the limit cases $\alpha=0$ and $\alpha=1$, the first three terms in equation \eqref{schrodinger}
lead to a cubic Schr\"odinger equation and to a linear heat equation, respectively.
The Cauchy problem with a Dirac delta for these kind of equations associated with a power type non-linearity has been studied
by several authors (see e.g. \cite{banica-vega1} and the reference therein).
We recall two  classical results.

\begin{thm}[\cite{brezis-friedman}]\label{brezis-friedman}
Let $p\geq 2$ and $u\in L^p_{\loc}(\R\times \R^+)$ be a solution in the sense of distributions of
\bq\label{heat-eq}
 \partial_t u-\partial_{xx}u+\abs{u}^{p}u=0 \quad \text{ on } \R\times \R^+.
 \eq
Assume that
\bq\label{cond-brezis}
\lim_{t\to0^+}\int_\R  u(x,t)\varphi(x)\,dx=0,\quad \text{  for all }\varphi \in C_0(\R\setminus\{0\}),
\eq
where $C_0(\R\setminus\{0\})$ denotes the space of continuous functions with compact support in $\R\setminus\{0\}$.
Then $u\in C^{2,1}(\R\times [0,\infty))$ and $u(x,0)=0$ for all $x\in\R$.
In particular there is no solution of \eqref{heat-eq} such that
\bqq
\lim_{t\to0^+}\int_\R u(x,t)\varphi(x)\,dx=\varphi(0),\quad \text{  for all }\varphi \in C_0(\R^N).
\eqq
\end{thm}
In \cite{brezis-friedman} it is also proved that if $1<p<2$, equation \eqref{heat-eq} has a
global solution with a Dirac delta as initial condition, as in the case of the linear parabolic
equation. Concerning the Schr\"odinger equation, we have the following
ill-posedness result due to Kenig, Ponce and Vega~\cite{kenig-ponce-vega}.
\begin{thm}[\cite{kenig-ponce-vega}]\label{thm-KPV}
Let $p\geq 2$. Either there is no
solution in the sense of distributions of
\bq\label{cubic-NLS}
 i\partial_t u+\partial_{xx}u+\abs{u}^{p}u=0 \quad \text{ on } \R\times \R^+,
 \eq
with
\bqq
\lim_{t\to 0^+}u(\cdot, t)=\delta \quad \text{ in }S'(\R),
\eqq
in the class $u,\abs{u}^pu\in L^\infty(\R^+;S'(\R))$,
or there is more than one.
\end{thm}
After performing an appropriate change of variables, equation \eqref{schrodinger} leads to the following equation
\bq\label{schrodinger-bis}
i\partial_t u+(\beta-i\alpha) \partial_{xx} u+\frac{u}{2}\left(\beta\abs{u}^2+2\alpha\int_0^x
\Im(\bar u \partial_y u)\,dy \right)=0.
\eq
Since $\alpha\in[0,1]$, the above equation can be seen as an intermediate model between \eqref{heat-eq} and \eqref{cubic-NLS}. Therefore one could expect that when $\alpha\in(0,1]$, the solutions of \eqref{schrodinger-bis} share similar properties to those established in 
Theorem~\ref{brezis-friedman} for the equation \eqref{heat-eq}. We will continue to observe that this is not necessarily the case. To this end, we need the following:
\begin{prop}\label{prop-NLS}
For all $\alpha\in[0,1]$ and for all $\c\in \C\setminus\{0\}$ the function  $w_{\c,\alpha} :\R\times \R^+\to\C$ given by
\bqq
w_{\c,\alpha}(x,t)=\frac{\c}{\sqrt{t} }\exp\left(
\frac{i\beta\abs{\c}^2}2 \ln(t)+ (i\beta-\alpha)\frac{x^2}{4t}
\right)
\eqq
is a solution of \eqref{schrodinger-bis}.
In addition,
\begin{enumerate}
 \item[i)] If $\alpha\in [0,1)$, then $w_{\c,\alpha}(\cdot, t)$ does not converge in $S'(\R)$ as $t\to0^+$.
\item[ii)] If $\alpha\in(0,1]$, then
\bq\label{conv-phi-0}
\lim_{t\to0^+}\int_\R  w_{\c,\alpha}(x,t)\varphi(x)\,dx=0,\quad \text{  for all }\varphi \in C_0(\R\setminus\{0\}).
\eq
\end{enumerate}
\end{prop}
\begin{proof}
A straightforward computation shows that $w_{\c,\alpha}$ satisfies \eqref{schrodinger-bis}.
The proof that  $w_{\c,\alpha}(\cdot, t)$ does not converge in $S'(\R)$ as $t\to0^+$ if $\c\neq 0$
is the same as in \cite{kenig-ponce-vega}. Indeed, for $\vp\in S(\R)$,
by Parseval's theorem,
\begin{align*}
\int_{\R} \bar w_{\c,\alpha}(x,t)\vp(x)\,dx&=
\frac{1}{2\pi}\int_{\R} \overline{\wh w}_{\c,\alpha}(\xi,t)  \wh\vp(\xi)\,d\xi\\
&=\frac{\bar \c e^{-i\beta\abs{\c}^2\ln(t)/2}}{\sqrt{\pi }}\overline{\sqrt{\alpha+i\beta}}  \int_\R e^{-(\alpha+i\beta)\xi^2t} {\wh\vp}(\xi)\,d\xi.
\end{align*}
By the dominated convergence theorem, the last integral converges:
 \bqq
 \lim_{t\to 0^+}\int_\R e^{-(\alpha+i\beta)\xi^2t} {\wh\vp}(\xi)\,d\xi=\int_\R  {\wh\vp}(\xi)\,d\xi=2\pi \vp(0).
 \eqq
Since $\beta\neq 0$, $e^{-i\beta\abs{\c}^2\ln(t)/2}$
does not admit a limit at $0$ in $S'(\R)$. We conclude that $w_{\c,\alpha}(\cdot, t)$ does not converge in $S'(\R)$
as $t\to0^+$.

It remains to prove \eqref{conv-phi-0}. Since now $\vp \in C_0(\R\setminus\{0\})$, we cannot proceed as before.
However, using the change of variables $x=\sqrt{t}y$, we have
\bqq
\lim_{t\to0^+}\int_\R  w_{\c,\alpha}(x,t)\varphi(x)\,dx=\c e^{i\beta\abs{\c}^2\ln(t)/2}\int_\R e^{(-\alpha+i\beta )y^2/4} \vp(\sqrt{t}y)\,dy.
\eqq
Therefore, since $\alpha>0$ and $\varphi(0)=0$, the  dominated convergence theorem implies that
\bqq
 \lim_{t\to 0^+}\int_\R e^{(-\alpha+i\beta )y^2/4} \vp(\sqrt{t}y)\,dy=\vp(0)\int_\R e^{(-\alpha+i\beta )y^2/4}\,dy =0.
 \eqq
 Since $\abs{e^{i\beta\abs{\c}^2\ln(t)/2}}=1$, we obtain \eqref{conv-phi-0}.
\end{proof}
The results in Proposition~\ref{prop-NLS} lead to the following remarks:
\begin{itemize}
\item[1. ] Observe that if $\alpha\in(0,1)$, $w_{\c,\alpha}$ provides a solution to the dissipative equation  \eqref{schrodinger-bis}.
Moreover, form part {\it{(ii)}}  in Proposition~\ref{prop-NLS}, $w_{\c,\alpha}$ satisfies the condition
\eqref{cond-brezis}. However, notice that  $w_{\c,\alpha}$ cannot be extended  to $C^{2,1}(\R\times [0,\infty))$
due to the presence of a logarithmic oscillation. This is in contrast with the properties for solutions of the cubic heat equation \eqref{heat-eq} established in Theorem~\ref{brezis-friedman}.
\item[2. ] In the case $\alpha=0$, equation \eqref{schrodinger-bis}
corresponds to \eqref{cubic-NLS} with $p=2$, i.e.\
to the equation cubic NLS equation that is invariant under the Galilean transformation. 
The proof of the ill-posedness result given in Theorem~\ref{thm-KPV} relies on this invariance and part {\it{(i)}} of Proposition~\ref{prop-NLS} with $\alpha=0$. 
Although when $\alpha>0$, equation \eqref{schrodinger-bis} is no longer invariant under the Galilean transformation, part {\it{(i)}}
of Proposition~\ref{prop-NLS} could be an indicator that
that the Cauchy problem \eqref{schrodinger-bis} with a delta as initial condition is still
ill-posed. This question rests open for the moment and it seems that the use of \eqref{schrodinger} (or \eqref{schrodinger-bis})
can be more difficult to formulate a Cauchy theory for \eqref{LLG} including self-similar solutions.
\end{itemize}

\section{Appendix}\label{appendix}
\setcounter{section}{1}
\def\thesection{\Alph{section}}
 \setcounter{equation}{0}
\setcounter{thm}{0}

\renewcommand\thethm{A.\arabic{thm}}

The characterization of $BMO_1^{-1}(\R^N)$ as sum of derivatives of functions in BMO
was proved by Koch and Tataru in \cite{koch-tataru}. A straightforward generalization of their proof
leads to the following characterization of $BMO_\alpha^{-1}(\R^N)$.

\begin{thm}\label{tataru}
Let $\alpha\in (0,1]$ and $f\in S'(\R^N)$.  Then $f\in BMO_\alpha^{-1}(\R^N)$
if and only if there exist $f_1,\dots,f_N\in BMO_\alpha(\R^N)$ such that $f=\sum_{j=1}^N \partial_j f_j$.
In addition, if such a decomposing holds, then 
$$\norm{f}_{BMO^{-1}_\alpha} \lesssim \sum_{j=1}^N [f_j]_{BMO_\alpha}.$$
\end{thm}

The next results provide the equivalence between  the weak solutions and the Duhamel formulation. We first need to introduce for $T>0$ the space
$L^1_{\uloc}(\R^N\times (0,T))$ defined as the space of measurable functions on $\R^N\times (0,T)$ such that the norm
\bqq
\norm{f}_{\uloc,T}:=\sup_{x_0\in \R^N}\int_{B(x_0,1)}\int_0^T\abs{f(y,t)}\,dt\,dy
\eqq
is finite. We refer the reader to Lemari\'e--Rieusset's book \cite{lemarie} for more details about these kinds of spaces.
In particular, we recall the following result corresponding to Lemma 11.3 in \cite{lemarie} in the case $\alpha=1$.
It is straightforward to check that the same proof still applies if $\alpha\in(0,1)$.
\begin{lemma}\label{lemma-equiv}
Let $\alpha\in (0,1]$, $T\in(0,\infty)$ and $w\in L^1_{\uloc}(\R^N\times (0,T))$.  Then the function
$$W(x,t):=\int_0^tS_{\alpha}(t-s)w(x,s)\,ds$$ is well defined and belongs to $L^1_{\uloc}(\R^N\times (0,T))$.
Moreover,
$$i\partial_{t}W+(\beta-i\alpha)\Delta W=w \quad \text{ in }\quad \mathcal \boD'(\R^N\times \R^+),$$
and the application
\begin{align*}
 [0,T]\ &\to \quad  \R\\
  t\quad &\mapsto \ \norm{W(\cdot,t)}_{L^1(B_1(x_0))}
 \end{align*}
is continuous for any $x_0\in\R$, with $\norm{W(\cdot,t)}_{L^1(B_1(x_0))}\to 0$, as $t\to 0^+$, uniformly in $x_0$.
\end{lemma}

Following the ideas in \cite{lemarie},
we can establish now the equivalence between the notions of solutions as well as the regularity.

\begin{thm}\label{thm-equiv} Let  $\alpha\in (0,1]$ and  $u\in X(\R^N\times \R^+;\C)$. Then the following assertions are
equivalent:
\begin{enumerate}
 \item[i)] The function $u$ satisfies
 \bq\label{eq-in-D}
 iu_t +(\beta-i\alpha)\Delta u=2(\beta-i\alpha)\frac{\bar u (\grad u)^2}{1+\abs{u}^2}\quad \text{ in }\quad \boD'(\R^N\times \R^+).
 \eq
 \item[ii)] There exists $u^0\in \boS'(\R^N)$ such that $u$ satisfies
 $$u(t)=S_\alpha(t) u^0
 -2(\beta-i\alpha) \int_0^t S_\alpha(t-s)\frac{\bar u (\grad u)^2}{1+\abs{u}^2}\,ds.$$
 \end{enumerate}
 Moreover, if (ii) holds, then $u\in C^\infty(\R^N\times \R^+)$ and
 \bq\label{lemma-CI}
\norm{ (u(t)-u^0)\vp}_{L^1(\R^N)}\to 0, \quad \textup{ as }t\to 0^+,
\eq
for any $\vp\in \boS(\R^N)$.
\end{thm}
\begin{proof}
In view of Lemma~\ref{lemma-equiv},  we need to prove that the function
$$g(u)=-2(\beta-i\alpha) \frac{\bar u (\grad u)^2}{1+\abs{u}^2}$$
belongs to $L^1_{\uloc}(\R^N\times (0,T))$, for all $T>0$. Indeed, by \eqref{est-g} we have
\bq\label{appen-1}
\norm{g(u)}_{\uloc,T}\leq \norm{\abs{\grad u}^2}_{\uloc,T}.
\eq
If $T\leq 1$, then
\bq\label{appen-2}
\norm{\abs{\grad u}^2}_{\uloc,T}\leq \sup_{x_0\in \R^N}\int_{Q_1(x_0)}\abs{\grad u(y,t)}^2 \,dt\,dy\leq \norm{u}_X^2.
\eq
If $T\geq 1$, using that
\bqq
\abs{\grad u}\leq \frac{[u]_{X}}{\sqrt t},\quad \text{ for any }t>0,
\eqq
we get
\bq\label{appen-3}
\begin{aligned}
\norm{\abs{\grad u}^2}_{T,\uloc}
&\leq \sup_{x_0\in \R^N}\int_{Q_1(x_0)}\abs{\grad u(y,t)}^2 \,dt\,dy+\sup_{x_0\in \R^N}\int_1^T\int_{B_1(x_0)} \abs{\grad u}^2\,dy \,dt\\
&\leq \norm{u}_X^2+[u]_X^2\abs{B_1(0)}\int_1^T \frac1t \,dt \\
&\leq  \norm{u}_X^2 (1+\abs{B_1(0)}\ln(T)).
\end{aligned}
\eq
In conclusion, we deduce from \eqref{appen-1}, \eqref{appen-2} and \eqref{appen-3}
that $g(u)\in  L^1_{\uloc}(\R^N\times (0,T))$ and then
it follows from Lemma~\ref{lemma-equiv} that (ii) implies (i).
The other implication can be established as in \cite[Theorem 11.2]{lemarie}.
Moreover, we deduce that the function
$$W(x,t):=T(g(u))(x,t)=\int_0^t S_{\alpha}(t-s)g(u)\,ds$$
satisfies  $\norm{W(\cdot,t)}_{L^1(B_1(x_0))}\to 0$, as $t\to 0^+$, uniformly in $x_0\in \R^N$.
Let us take $\vp\in S(\R^N)$ and a constant $C_\vp>0$ such that
$\abs{\vp(x)}\leq C_\vp (2+\abs{x})^{-N-1}$. Then
\begin{align*}
\int_{\R^N} \abs{\vp(y)W(y,t)}\,dy &\leq \sum_{k\in \Z^N}\int_{B_1(k)} \frac {C_\vp}{(2+\abs{x})^{N+1}}\abs{W(y,t)}\,dy\\
&\leq \sup_{x_0\in \R^N}  \norm{W(\cdot,t)}_{L^1(B_1(x_0))} \sum_{k\in \Z^N} \frac {C_\vp}{(1+\abs{k})^{N+1}},
\end{align*}
so that
$\norm{\vp W(\cdot,t)}_{L^1(\R^N)}\to 0$ as $t\to 0^+$, i.e.
 \bq\label{proof-cont1}
 \norm{(u(t)-S_\alpha(t)u^0)\vp}_{L^1(\R^N)}\to 0, \quad \textup{ as }t\to 0^+.
 \eq
On the other hand,  since $u^0\in L^\infty(\R^N)$,
\bq\label{limit0}
\norm{S_{\alpha}(t)u^0-u^0}_{L^1(B_r(0))}\to 0, \quad \textup{ as }t\to 0^+,
\eq
for any $r>0$ (see e.g. \cite[Corollary 2.4]{arrieta}). Given $\epsilon>0$, we fix $r_\epsilon>0$ such that
$$2\norm{u^0}_\infty\norm{\vp}_{L^1(B^c_{r_\epsilon}(0))}\leq \epsilon.$$
Using \eqref{limit0}, we obtain
\bqq
\lim_{t\to0^+}\norm{ (S_{\alpha}(t)u^0-u^0)\vp}_{L^1(B_{r_\epsilon}(0))}=0.
\eqq
Then, passing to limit in the inequality
\begin{align}\label{proof-CI}
\norm{ (S_{\alpha}(t)u^0-u^0)\vp}_{L^1(\R^N)}
\leq \norm{ (S_{\alpha}(t)u^0-u^0)\vp}_{L^1(B_{r_\epsilon}(0))}+2\norm{u^0}_{L^\infty(\R^N)}\norm{\vp}_{L^1(B^c_{r_\epsilon}(0))},
\end{align}
we obtain
\bq\label{proof-cont2}
\limsup_{t\to 0^+}\norm{ (S_{\alpha}(t)u^0-u^0)\vp}_{L^1(\R^N)}\leq \epsilon.
\eq
Therefore
$$\lim_{t\to 0^+}\norm{ (S_{\alpha}(t)u^0-u^0)\vp}_{L^1(\R^N)}=0.$$
Combining with  \eqref{proof-cont1},  we conclude  the proof of \eqref{lemma-CI}.

It remains to prove that $u$ is smooth for $t>0$.
Since  $u\in X(\R^N\times \R^+;\C)$, we get that
 $u,\grad u\in  L^\infty_{\loc}(\R^N\times \R^+)$. Then $g(u)\in L^2_\loc(\R^N\times\R^+)$
so the $L^p$-regularity theory for parabolic equations implies that a  function $u$
satisfying \eqref{eq-in-D} belongs to $u\in H_{\loc}^{2,1}(\R^N\times\R^+)$ (see \cite{lieberman,lady-parabolic} and \cite[Remark 48.3]{quittner}
for notations and more details). Since  the space $H^k\cap L^\infty$ is stable under multiplication (see e.g.\ \cite[Chapter 6]{Horm97}),
we can use a bootstrap argument to conclude that $u\in C^\infty(\R^N\times \R^+)$.
\end{proof}

\begin{remark}
Several authors have studied further properties  of the solutions found by Koch and Tataru for the Navier--Stokes equations.
For instance,  analyticity,   decay rates of the higher-order derivatives in space and time have been investigated by
Miura and Sawada~\cite{miura}, Germain, Pavlovi{\'c} and Staffilani~\cite{germain-pav}, among others.
A similar analysis for the solution $u$ of \eqref{DNLS} is beyond the scope of this paper, but
it can probably be performed  using the same arguments given in  \cite{miura,germain-pav}.
\end{remark}

We end this appendix with some properties of the self-similar found in \cite{gutierrez-delaire}.
\begin{thm}[\cite{gutierrez-delaire}]\label{thm-self} Let $N=1$. For every $\alpha\in[0,1]$ and  $c>0$, there exists
a profile $\f_{c,\alpha}\in C^\infty(\R,\S^2)$ such that
$$
\m_{c, \alpha}(x,t) =\f_{c,\alpha}\left( \frac{x}{\sqrt{t}}  \right),   \qquad \text{ for all }(x,t)\in \R\times \R^+,
$$
is a  smooth solution of \eqref{LLG}  on $\R\times \R^+$.
Moreover,
\begin{enumerate}[label=({\roman*}),ref={{({\roman*})}}]
\item\label{converge} There exist unitary vectors $\A^{\pm}_{c,\alpha}=(A_{j,c,\alpha}^{\pm})_{j=1}^{3} \in \S^2$ such that the following pointwise convergence holds when $t$ goes to zero:
\bq\label{convergencia} \lim_{t\to
 0^+}\m_{c,\alpha}(x,t)=\begin{cases}
 \A^+_{c,\alpha}, &\text{ if }x>0,\\[2ex]
 \A^-_{c,\alpha}, &\text{ if }x<0,
\end{cases}
\eq
and $\A^-_{c,\alpha}=(A_{1,c,\alpha}^+,-A_{2,c,\alpha}^+,-A_{3,c,\alpha}^+)$.
\item\label{rate} There exists a constant $C(c,\alpha,p)$, depending only on $c$, $\alpha$ and $p$ such that for
all $t>0$
\bq\label{cota-m} \|\m_{c,\alpha}(\cdot,t)-
\A^+_{c,\alpha}\chi_{(0,\infty)}(\cdot)-
\A^{-}_{c,\alpha}\chi_{(-\infty,0)}(\cdot)\|_{L^p(\R)}\leq
 C(c,\alpha,p) t^\frac{1}{2p}, \eq for all $p\in (1,\infty)$. In
addition, if $\alpha>0$, \eqref{cota-m} also holds for $p=1$.
\item\label{item-derivada} For $t>0$ and $x\in \R$, the derivative in space satisfies
\bq\label{derivada}
\abs{\partial_{x}\m_{c,\alpha} (x,t)}=\frac{c}{\sqrt t}e^{-\frac{\alpha x^2}{4t}}.
\eq
\item\label{cotas-A}
Let $\alpha\in [0,1]$.
Then $\A^+_{c,\alpha}\to (1,0,0)$ as $c\to 0^+$.
\end{enumerate}
\end{thm}

\begin{lemma}\label{lemma-theta}
Let $c> 0$, $\alpha\in (0,1]$,   $\A^{+}_{c,\alpha}, \A^{-}_{c,\alpha}$ be the unit vectors
given in Theorem~\ref{thm-self} and $\vartheta_{c,\alpha}$ the angle between  $\A^{+}_{c,\alpha}$
and  $\A^{-}_{c,\alpha}$. Then, for fixed $\alpha\in (0,1]$, $\vartheta_{c,\alpha}$ is a continuous
function in $c$. Also, for $0<c<\alpha^2\sqrt{\pi}/32$,
\bq\label{est-theta}
\vartheta_ {c,\alpha}\geq \arccos\left(1-c^2\pi+\frac{32c^3\sqrt{\pi}}{\alpha^2}\right).
\eq
\begin{remark}
If $\alpha\in (0,1]$ and $c\in(0,\alpha^2\sqrt{\pi}/32)$, then  $1-c^2\pi+\frac{32c^3\sqrt{\pi}}{\alpha^2}\in (0,1)$,
so that its $\arccos$  is well-defined.
\end{remark}

\end{lemma}
\begin{proof}
The continuity was proved in  \cite{gutierrez-delaire}. To show the estimate \eqref{est-theta}, 
 we use Theorem~1.3 in \cite{gutierrez-delaire}, to get
$$\left|A_{2, c,\alpha}^+-c\frac{\sqrt{\pi(1+\alpha)}}{\sqrt2}\right|\leq \frac{c^2\pi}{4}
  +\frac{c^2\pi}{\alpha\sqrt{2}}\left(1+\frac{c^2 \pi}{8}+c\frac{\sqrt{\pi(1+\alpha)}}{2\sqrt2}\right)
  +\left(\frac{c^2\pi}{2\sqrt{2}\alpha}\right)^2.$$
Since $\alpha\in(0,1]$, we have for any $c\in (0,1]$,
  $$\frac{\pi}{4}
  +\frac{\pi}{\alpha\sqrt{2}}\left(1+\frac{c^2 \pi}{8}+c\frac{\sqrt{\pi(1+\alpha)}}{2\sqrt2}\right)
  +\frac{c^2\pi^2}{8\alpha^2}
  \leq
\frac{1}{\alpha^2}\left(
\frac{\pi}{4}+\frac{\pi}{\sqrt 2}\left(1+\frac{\pi}{8}+\frac{\sqrt{\pi}}{2}\right)+\frac{\pi^2}{8}
\right)
\leq \frac{8}{\alpha^2}.$$
We deduce that for all $\alpha,c\in(0,1]$,
\bqq
A_{2, c,\alpha}^+\geq c\frac{\sqrt{\pi(1+\alpha)}}{\sqrt2}-\frac{8c^2}{\alpha^2}\geq c\frac{\sqrt{\pi}}{\sqrt2}-\frac{8c^2}{\alpha^2}.
\eqq
In particular $A_{2, c,\alpha}^+\geq 0$ if $c\leq \alpha^2\sqrt{\pi}/(8\sqrt{2})$. Thus
\bqq
(A_{2, c,\alpha}^+)^2\geq \frac{c^2\pi}2-\frac{16c^3\sqrt{\pi}}{\sqrt{2}\alpha^2}+\frac{64c^4}{\alpha^4}\geq \frac{c^2\pi}2-\frac{16c^3\sqrt{\pi}}{\alpha^2},
\eqq
so that
\bqq
\cos(\vartheta_{c,\alpha})=1-2((A_{2, c,\alpha}^+)^2+(A_{3, c,\alpha}^+)^2)\leq
1-c^2\pi+\frac{32c^3\sqrt{\pi}}{\alpha^2},
\eqq
which implies \eqref{est-theta}.
\end{proof}

The following lemma is a slightly refinement of Theorem~1.4 in \cite{gutierrez-delaire}.
\begin{lemma}[\cite{gutierrez-delaire}]\label{lemma-A}
Let $c> 0$, $\alpha\in[0,1]$ and  $\A^{+}_{c,\alpha}$ be the unit vector
given in Theorem~\ref{thm-self}. Then $\A^+_{c,\alpha}$ is a continuous function of $\alpha$ in $[0,1]$
and
\bq    \label{thm-alpha2}
 \abs{\A^+_{c,\alpha}-\A^+_{c,1}} \leq h(c)\, \sqrt{1-\alpha}, \quad \text{for all }\ \alpha\in [1/2,1],
   \eq
where $h:\R^+\to\R^+$ is a strictly increasing function satisfying
$$\lim_{s\to\infty}h(s)=\infty.$$
\end{lemma}
\begin{proof}
  In view of  \cite[Theorem~1.4]{gutierrez-delaire}, we only need to prove that the constant $C(c)$
 in the statement of the Theorem~1.4 (notice that $c_0$ in \cite{gutierrez-delaire} corresponds to $c$ in our notation)
 is  polynomial in $c$ with nonnegative coefficients.
 Looking at the proof of  \cite[Theorem~1.4]{gutierrez-delaire}, we see that the constant $C(c)$
behaves like the constant in inequality (3.108) in \cite{gutierrez-delaire}.
In view of (3.17),  the estimate (3.23) in \cite{gutierrez-delaire} can be written as
$$\abs{f(s)}\leq \sqrt{2} \quad \text{and}\quad \abs{f'(s)}\leq \frac{c}{2}e^{-\alpha s^2/4},$$
and then (3.18) can be recast as
$$\abs{g}\leq \left(\frac{c}4+\frac{c^2\sqrt{2}}8
\right)\left(\frac{s}{\beta}e^{-\alpha
s^2/4}+s^2e^{-\alpha s^2/2}\right).$$
Then, it can be easily checked that the function $h$ is a polynomial with nonnegative coefficients.
\end{proof}

\begin{thank}
A.~de~Laire was partially  supported  by the Labex CEMPI (ANR-11-LABX-0007-01) and  the MathAmSud program. S.~Gutierrez was partially supported by the EPSRC grant EP/J01155X/1 and the ERCEA Advanced Grant 2014 669689 - HADE.
\end{thank}

 \bibliography{ref}
 \bibliographystyle{abbrv}

\end{document}